\documentclass[10pt]{article}
\usepackage{amssymb,amsmath,amsfonts,amsthm,enumerate,color}
\usepackage[toc,page,title,titletoc,header]{appendix}
\usepackage{graphicx}
\usepackage{indentfirst}
\usepackage{multicol}
\usepackage{booktabs}

\numberwithin{equation}{section}

\newtheorem{Theorem}{Theorem}[section]
\newtheorem{Lemma}[Theorem]{Lemma}
\newtheorem{Proposition}[Theorem]{Proposition}

\newtheorem{Corollary}[Theorem]{Corollary}

\numberwithin{equation}{section}

\def\no{\noindent} \def\p{\partial} 
\def \Vh0{\stackrel{\circ}{V}_h} \def\to{\rightarrow}

\newcommand{\q}{\quad}   \newcommand{\qq}{\qquad} 
\def\l{\label}  \def\f{\frac}  

\def\m{\mbox}   
\def\ms{\medskip}  \def\ss{\smallskip}

\def\p{\partial}

\newcommand{\lc}
{\mathrel{\raise2pt\hbox{${\mathop<\limits_{\raise1pt\hbox
{\mbox{$\sim$}}}}$}}}

\newcommand{\gc}
{\mathrel{\raise2pt\hbox{${\mathop>\limits_{\raise1pt\hbox{\mbox{$\sim$}}}}$}}}

\newcommand{\ec}
{\mathrel{\raise2pt\hbox{${\mathop=\limits_{\raise1pt\hbox{\mbox{$\sim$}}}}$}}}

\def\bb{\begin{equation}} \def\ee{\end{equation}}

\def\beqn{\begin{eqnarray}}  \def\eqn{\end{eqnarray}}

\def\beqnx{\begin{eqnarray*}} \def\eqnx{\end{eqnarray*}}

\def\bn{\begin{enumerate}} \def\en{\end{enumerate}}

\def\bd{\begin{description}} \def\ed{\end{description}}

\makeatletter
\newenvironment{tablehere}
  {\def\@captype{table}}
  {}
\newenvironment{figurehere}
  {\def\@captype{figure}}
  {}
\makeatother

\title{Optimal Shape Design by Partial Spectral Data}

\author{
Habib Ammari\footnote{Department of Mathematics and Applications, Ecole Normale Sup$\acute{\text{e}}$rieure, 45 Rue d'Ulm, 75005 Paris, France.
The work of this author
was supported by ERC Advanced Grant Project MULTIMOD--267184.
(habib.ammari@ens.fr).}
\and Yat Tin Chow\footnote{Department of Mathematics, Chinese University of Hong Kong, Shatin, N.T., Hong Kong (ytchow@math.cuhk.edu.hk, kjliu@math.cuhk.edu.hk).}
\and Keji Liu\footnotemark[3]
\and Jun Zou\footnote{The work of this author was substantially supported by Hong Kong RGC grants
(projects 405513 and 404611). (zou@math.cuhk.edu.hk)}
}

\begin{document}

\date{}
\maketitle

\begin{abstract}
In this paper, we are concerned with a shape design problem, in which our target is to design, up to rigid transformations and scaling, the shape of an object given either its  polarization tensor at multiple contrasts or
the partial eigenvalues of its Neumann-Poincar\'e  operator, which are known as the Fredholm eigenvalues. We begin by proposing
to recover the eigenvalues of the Neumann-Poincar\'e operator from the polarization tensor by means
of the holomorphic functional calculus.  Then we develop a
regularized Gauss-Newton optimization method for the shape reconstruction process.
We present numerical results to demonstrate the effectiveness of the proposed methods and to illustrate important properties of the Fredholm eigenvalues and their associated eigenfunctions.
Our results are expected to have important applications in the design of plasmon resonances in nanoparticles as well as in the multifrequency or pulsed imaging of small anomalies.
\end{abstract}

\bigskip



%

\section{Introduction}
Fredholm eigenvalues are the eigenvalues of the integral Neumann-Poincar\'e operator, which arises naturally in solving Neumann transmission problems for the Laplacian.  They depend on the shape of the domain but they are invariant under rigid transformations and scaling.
They
have been the subject of intensive investigations; see, for instance, \cite{ahlfors,schiffer1,schiffer2,springer}. Spectral analysis of Neumann-Poincar\'e type operators has played a key role in
the mathematical justification of cloaking due to
  anomalous localized resonance  \cite{milton} and in the analysis of gradient blow-up phenomena in the presence of nearly touching inclusions \cite{ciraolo, triki2, triki}.  We also refer to \cite{shapiro} where new and interesting facts on spectral analysis
related to the Neumann-Poincar\'e integral operator have been obtained and to the works on plasmon resonances \cite{plasmon4,plasmon1,plasmon3}. Plasmon resonant nanoparticles such as gold nanoparticles offer, in addition to their biocompatibility, enhanced scattering and absorption, making them  not only suitable for use as a contrast agent but also in therapeutic applications \cite{plasmon4}. Recently, it has been shown that plasmon resonances in nanoparticles can be treated as an eigenvalue problem for the Neumann-Poincar\'e operator, which leads to direct calculation of resonance values of permittivity and resonance frequency \cite{grieser, plasmon3}.
In biomedical applications, it is challenging to design nanoparticles that resonate at specified frequencies. It is the purpose of the paper to propose an efficient approach for solving the optimal design problem (up to rigid transformations and scaling) from partial Fredholm eigenvalues.

Shape identification  from Fredholm eigenvalues has also important applications in imaging. In electrosensing,
 the polarization tensor (PT) of a target at multiple frequencies (or equivalently at multiple contrasts) can be reconstructed from electrical capacitance measurements \cite{boulier3,boulier1,boulier2,seo1,seo2}. The PT
arises naturally when we describe the perturbation of the electrical potential due to the presence of the target whose admittivity is different from that of the background.
In fact, the polarization tensor of an inclusion can be expressed in terms of the Neumann-Poincar\'e operator and the admittivity contrast.

In this paper, we first show that the Fredholm eigenvalues can be reconstructed from the polarization tensor at multiple contrasts. By doing so, we connect design problems for plasmon resonances in nanoparticles to the imaging of small anomalies.  Moreover, we show how to
obtain in practice the polarization tensor at multiple contrasts from electrical capacitance tomography measurements. By probing the domain with an electric pulse, the polarization tensor of the anomaly at multiple frequencies and therefore at multiple contrasts can be recovered
  \cite{pulse2, pulse}; see Appendix \ref{curve}. We optimize the pulse shape in order to reconstruct in the most stable way the first few Fredholm eigenvalues.

Then we consider the shape reconstruction problem (up to rigid transformations and scaling), in which we wish to reconstruct a shape from only the prior knowledge of the first several Fredholm eigenvalues of the Neumann-Poincar\'e operator. We start by giving both analytical and numerical evidence that the first Fredohlm eigenvalues contain only low-frequency information about the shape of the domain while higher ones contain higher frequency information. We estimate the oscillation behavior of the associated eigenfunctions. We also emphasize the exponential decay of the Fredholm eigenvalues in the two dimensional case. This clearly makes the design problem exponentially ill-posed. Therefore, we should restrict ourselves to low-frequency shape reconstructions from the few first Fredholm eigenvalues.

We also derive Hadamard's formula for the Fredholm eigenvalues. Based on Osborn's theorem \cite{osborn}, we compute the shape derivative of Fredholm eigenvalues using the shape derivative of the Neumann-Poincar\'e operator. Then we propose a minimization algorithm to reconstruct a domain given its first Fredholm eigenvalues. In view of the invariance of the Fredholm eigenvalues under rigid transformations and scaling, we incorporate
some effective penalty and regularization terms in the cost functional to ensure the local existence and uniqueness of its minimizers. We will further present several numerical illustrations of our main findings.

Our results on Fredholm eigenvalues and on the polarization tensor are expected to have important applications not only in shape design problems but also in shape classification and recognition problems. Various other geometric quantities associated with the shape of a domain, such as eigenvalues, capacities, harmonic moments, and generalized polarization tensors are used to distinguish between objects and classify them \cite{boulier1, booknew, yu, zribi, ben, flusser_moments_2009, hu_visual_1962, liao_image_1996}. The concept of polarization tensor at
multiple contrasts seems to be the most natural one for shape classification and recognition using capacitance electrical impedance tomography.

The paper is organized as follows. In section \ref{sec2},  we introduce the Neumann-Poincar\'e operator and the concept of polarization tensor associated with a given domain and a given contrast. In section \ref{GPT},  two methods are provided for reconstructing Fredholm eigenvalues of a domain from its polarization tensor at all contrasts, then tested numerically.
Section \ref{sec4} is devoted to the derivation of a Hadamard's perturbation formula for Fredholm eigenvalues. By combining the results in \cite{zribi} on the shape derivative of the Neumann-Poincar\'e operator together with Osborn's theorem \cite{osborn}, we compute the shape derivative of Fredholm eigenvalues. In section \ref{shape}, we present and numerically test our minimization procedure for finding low-frequency features of a domain from its first few Fredholm eigenvalues. In Appendix \ref{curve}, we show the method to obtain the polarization tensors at multiple contrasts from electrical capacitance tomography measurements. In Appendix \ref{multipleobject}, we consider the case of multiply connected objects. In that case, it is remarkable to easily find the number of connected components from the multiplicity of the Fredholm eigenvalues.

\section{Neumann-Poincar\'e operator and polarization tensor} \label{sec2}
In this section, we first introduce the Neumann-Poincar\'e operator of an open connected domain $D$ with $\mathcal{C}^2$ boundary in $\mathbb{R}^d\;(d=2,3)$.  Given such a domain $D$, we consider the following Neumann problem,
\beqn
        \Delta u = 0  \q \text{ in } ~~D\,; \qq
        \frac{\partial u}{\partial \nu} = g  \q \text{ on } ~~\partial D,
        \qq \int_{\partial D} u \, d\sigma =0, \label{neumann}
\eqn
where $g \in L^2_0(\partial D)$ with $L^2_0(\partial D)$ being the set of functions in $L^2(\partial D)$ with zero mean-value. In (\ref{neumann}),   $\partial/\partial \nu$ denotes the normal derivative.
We note that the Neumann problem (\ref{neumann}) can be rewritten as a boundary integral equation with the help of the single-layer potential. Given a density function $\phi \in L^2(\partial D)$, the single-layer potential, $\mathcal{S}_{\partial D} [\phi]$, can be defined as follows,
\beqn
    \mathcal{S}_{\partial D} [\phi] (x) := \int_{\partial D} \Gamma(x-y) \phi(y) d \sigma(y)
\eqn
for $x \in \mathbb{R}^d$, where $\Gamma$ is the fundamental solution of the Laplacian in $\mathbb{R}^d$ :
\beqn
    \Gamma (x-y) =
    \begin{cases}
     -\f{1}{2\pi} \log |x-y| & \text{ if }\; d = 2 \, ,\\
     \f{1}{(2-d)\omega_d} |x-y|^{2-d} & \text{ if }\; d > 2 \, ,
    \end{cases}
    \label{fundamental}
\eqn
where $\omega_d$ denotes the surface area of the unit sphere in $\mathbb{R}^d$.  It is well-known that the single-layer potential satisfies the following jump condition on $\partial D$:
\beqn
    \f{\p}{\p \nu} \left(  \mathcal{S}_{\partial D} [\phi] \right)^{\pm} = (\pm \f{1}{2} I + \mathcal{K}^*_{ \partial D} )[\phi]\,,
    \label{jump_condition}
\eqn
where the superscripts $\pm$ indicate the limits from outside and inside $D$ respectively, and 
$\mathcal{K}^*_{\partial D}: L^2(\partial D) \rightarrow L^2(\partial D)$ is the Neumann-Poincar\'e operator defined by
\beqn
    \mathcal{K}^*_{\partial D} [\phi] (x) := \f{1}{\omega_d} \int_{\partial D} \f{\langle  x-y,\nu_x \rangle  }{|x-y|^d} \phi(y) d \sigma(y) \,, 
    \label{operatorK}
\eqn
with $\nu_x$ being 
the outward normal at $x \in \partial D$. We note that $\mathcal{K}^*_{\partial D}$ maps $L^2_0(\partial D)$ onto itself.

With these notions, the Neumann problem (\ref{neumann}) can then be formulated as
\beqn
    g = \f{\p}{\p \nu} \left(  \mathcal{S}_{\partial D} [\phi] \right)^{-} =  ( - \f{1}{2} I + \mathcal{K}^*_{ \partial D} )[\phi]  \, .
\eqn
Therefore, the solution to the Neumann problem (\ref{neumann}) can be reformulated as a solution to the boundary integral equation with the Neumann-Poincar\'e operator $\mathcal{K}^*_{ \partial D}$.

The operator $\mathcal{K}^*_{ \partial D}$ arises not only in solving the Neumann problem for the Laplacian but also for representing the solution to the transmission problem as described below.

Consider an open connected domain $D$ with $\mathcal{C}^2$ boundary in $\mathbb{R}^d$.  Given a harmonic function $u_0$ in $\mathbb{R}^d$, we consider the following transmission problem in $\mathbb{R}^d$:
\beqn
    \begin{cases}
        \nabla \cdot (\varepsilon_{D} \nabla u) = 0 &\text{ in }\; \mathbb{R}^d, \\[1.5mm]
         u - u_0 = O(|x|^{1-d}) &\text{ as }\; |x| \rightarrow \infty,
    \end{cases}
    \label{transmission}
\eqn
where $\varepsilon_{D}  = \varepsilon_c \chi(D) +  \varepsilon_m \chi(\mathbb{R}^d \backslash \overline{D})$ with $\varepsilon_c, \varepsilon_m$ being two positive constants, and $\chi(\Omega)$ is the characteristic function of the domain $\Omega= D$ or $\mathbb{R}^d \backslash \overline{D}$.  With the help of the single-layer potential, we can rewrite the perturbation  $u - u_0$, which is  due to the inclusion $D$, as
\beqn
    u - u_0 = \mathcal{S}_{\partial D} [\phi] \, ,
    \label{scattered}
\eqn
where $\phi \in L^2(\partial D)$ is an unknown density, and $\mathcal{S}_{\partial D} [\phi]$ is the refraction part of the potential in the presence of the inclusion. The transmission problem (\ref{transmission}) can be rewritten as
\beqn
    \begin{cases}
        \Delta u = 0 &\text{ in }\; D \bigcup (\mathbb{R}^d\backslash \overline{D} ) \, ,\\[1.5mm]
        u^+ = u^- &\text{ on }\; \partial D \, ,\\[1.5mm]
        \varepsilon_c \f{\p u^{+}}{\p \nu}  =  \varepsilon_m \f{\p u^{-}}{\p \nu} &\text{ on }\; \partial D \, ,\\[1.5mm]
        u - u_0 = O(|x|^{1-d}) &\text{ as }\; |x| \rightarrow \infty \, .
    \end{cases}
    \label{transmission2}
\eqn
With the help of the jump condition (\ref{jump_condition}), solving the above system (\ref{transmission2})
can be regarded as solving the density function $\phi \in L^2(\partial D)$ of the following integral equation
\beqn
    \frac{\partial u_0}{\partial \nu} =  \left( \f{\varepsilon_c+\varepsilon_m}{2(\varepsilon_c-\varepsilon_m)}I - \mathcal{K}_{\partial D}^* \right) [\phi] \, .
    \label{potential2}
\eqn
With the harmonic property of $u_0$, we can write
\beqn
    u_0 (x) = \sum_{\alpha \in \mathbb{N}^d } \f{1}{\alpha !} \partial^\alpha u_0(0) x^\alpha
\label{series}
\eqn
with $\alpha=(\alpha_1,\ldots, \alpha_d) \in \mathbb{N}^d,\, \partial_\alpha = \partial_1^{\alpha_1} \ldots \partial_d^{\alpha_d}$ and $\alpha!=\alpha_1!\ldots \alpha_d!$\,.

Consider $\phi^\alpha$ as the solution of the Neumann-Poincar\'e operator:
\beqn
   \frac{\partial x^\alpha}{\partial \nu} =  \left( \f{\varepsilon_c+\varepsilon_m}{2(\varepsilon_c-\varepsilon_m)}I - \mathcal{K}_{\partial D}^* \right) [\phi^\alpha] \, .
    \label{phi_alpha}
\eqn
The invertibilities of the operator $( \f{\varepsilon_c+\varepsilon_m}{2(\varepsilon_c-\varepsilon_m)}I - \mathcal{K}_{\partial D}^*)$
from $L^2(\partial D)$ onto $L^2(\partial D)$ and from
$L_0^2(\partial D)$ onto $L_0^2(\partial D)$ are proved, for example, in \cite{book, kellog},
provided that $|\f{\varepsilon_c+\varepsilon_m}{2(\varepsilon_c-\varepsilon_m)}| > 1/2$.
We can substitute (\ref{series}) and (\ref{phi_alpha}) back into (\ref{scattered}) to get
\beqn
    u - u_0 = \sum_{|\alpha| \geq 1 }  \f{1}{\alpha !} \partial^\alpha u_0(0) \mathcal{S}_{\partial D} [\phi^\alpha] = \sum_{|\alpha| \geq 1 } \f{1}{\alpha !} \partial^\alpha u_0(0) \int_{\partial D} \Gamma(x-y) \phi^\alpha (y) d \sigma(y) \, .
    \label{scattered2}
\eqn
Using the Taylor expansion,
\beqn
\Gamma (x-y) = \Gamma(x) - y \cdot \nabla \Gamma(x) + O(\frac{1}{|x|^{d}})\, ,
\label{series2}
\eqn
which holds for all $x$ such that $|x|\rightarrow \infty$ while $y$ is bounded \cite{book},
we get the following result by substituting (\ref{series2}) into (\ref{scattered2}) that
\beqn \label{eq15}
    (u - u_0)(x) = \nabla u_0(0) \cdot M(\lambda, D) \nabla \Gamma (x) + O(\frac{1}{|x|^{d}}) \quad \m{as} ~~|x| \rightarrow \infty,
\eqn
where $M=(m_{ij})_{i,j=1}^d$ is the polarization tensor (PT) associated with the domain $D$ and the contrast $\lambda$ defined by
\beqn
    m_{ij}(\lambda,D) :=  \int_{\partial D} y_i (\lambda I - \mathcal{K}_{\partial D}^*)^{-1} \left[\nu_j\right] (y) d \sigma(y) \, ,
    \label{polarizationtensionr}
\eqn
with $\lambda :=  \f{\varepsilon_c+\varepsilon_m}{2(\varepsilon_c-\varepsilon_m)} $\, and $\nu_j$ being the $j$-th component of $\nu$. Here we have used in (\ref{eq15}) the fact that
$\int_{\partial D} \nu \, d\sigma=0$.

Typically the constants $\varepsilon_c$ and $\varepsilon_m$ are positive in order to make the system (\ref{transmission2}) physical. This corresponds to the situation with $|\lambda| > \f{1}{2}$.

However, recent advances in nanotechnology make it possible to produce noble metal nanoparticles
with negative permittivities at optical frequencies \cite{plasmon4, plasmon2}. Therefore, we can have the possibility for some frequencies that $\lambda := \f{\varepsilon_c+\varepsilon_m}{2(\varepsilon_c-\varepsilon_m)}$ actually lies in the spectrum of $ \mathcal{K}_{\partial D}^*$.

If this happens, the following integral equation
\beqn
    0 =  \left( \lambda I - \mathcal{K}_{\partial D}^* \right) [\phi]   \quad \text{ on }\; \partial D \label{plasmonicreso}
\eqn
has non-trivial solutions $\phi \in L^2(\partial D)$ and the nanoparticle resonates at those frequencies.

Therefore, we have to investigate the mapping properties of the Neumann-Poincar\'e operator. Assume that  $\partial D$ is of class $\mathcal{C}^{1,\alpha}$. It is known that the operator $\mathcal{K}_{\partial D}^*: L^2(\partial D) \rightarrow L^2(\partial D)$ is compact \cite{kellog}, and its spectrum is discrete and accumulates at zero. All the eigenvalues are real and bounded by $1/2$. Moreover, $1/2$ is always an eigenvalue and its associated eigenspace is of dimension one,
which is nothing else but the kernel of the single-layer potential $\mathcal{S}_{\partial D}$.
In two dimensions, it can be proved that if $\lambda_i\neq 1/2$ is an eigenvalue of $\mathcal{K}_{\partial D}^*$, then $-\lambda_i$ is an eigenvalue as well. This property is known as the twin spectrum property; see \cite{plasmon1}. The Fredholm eigenvalues are the eigenvalues of $\mathcal{K}_{\partial D}^*$. It is easy to see, from the properties of $\mathcal{K}_{\partial D}^*$, that they are invariant with respect to rigid motions and scaling.
They can be explicitly computed for ellipses and spheres.  If $a$ and $b$ denote the semi-axis lengths of an ellipse then it can be shown that $\pm (\frac{a-b}{a+b})^i$ are its Fredholm eigenvalues \cite{shapiro}. For the sphere, they are given by $1/(2(2i+1))$; see \cite{seo}. It is worth noticing that the convergence to zero of Fredholm eigenvalues is exponential for ellipses while it is algebraic for spheres.

Equation (\ref{plasmonicreso}) corresponds to the case when plasmonic resonance occurs in $D$; see \cite{grieser}. The optimal shape design for Fredholm eigenvalues is of great interest in plasmonics \cite{plasmon4, plasmon1,plasmon2}.
Given negative values of $\varepsilon_c$, we show in this paper how to  design a shape with prescribed plasmonic resonances.

\section{Reconstruction of Fredholm eigenvalues from the polarization tensor} \label{GPT}

\subsection{Reconstruction method via holomorphic functional calculus} \label{func}
In this subsection we propose for two dimensions to recover the Fredholm  eigenvalues from the
polarization tensor
\beqn
    M (\lambda,D) :=  \int_{\partial D} y (\lambda I - \mathcal{K}_{\partial D}^*)^{-1} [\nu](y) d \sigma(y) \,
    \label{polarizationmatrix}
\eqn
for $\lambda$ along a simple closed curve $\gamma$ by
means of the holomorphic functional calculus.
From this expression, one observes that $M (\lambda,D)$ actually encodes vast information of the resolvent of the operator $\mathcal{K}_{\partial D}^*$ at $\lambda$,
\beqn
R_{\lambda}(\mathcal{K}_{\partial D}^*) := (\lambda I - \mathcal{K}_{\partial D}^*)^{-1} \, .
\label{resolvent}
\eqn

Motivated by this observation, we propose to recover eigenvalues $\{\lambda_i\}_{i \geq 1}$ ($\lambda_i\neq 1/2$) of $\mathcal{K}_{\partial D}^*$ from $M (\lambda,D)$ via the holomorphic functional calculus of $\mathcal{K}_{\partial D}^*$\,.
Let $\mathcal{H}$ be the space $L^2_0 (\partial D)$ equipped with the inner product
$-\langle \cdot, \mathcal{S}_{\partial D}(\cdot) \rangle_{L^2(\partial D)}$. Since $\mathcal{S}_{\partial D}$ is injective on $L^2_0(\partial D)$, $L^2_0(\partial D)$ is complete for this inner product.
If $\partial D$ is of class $\mathcal{C}^{1,\alpha}$,
 there exists a complete orthonormal set $\{\phi^\pm_i\}_{i \geq 1}$ in $\mathcal{H}$ such that $\mathcal{K}_{\partial D}^* \phi^\pm_i = \pm \lambda_i \phi^\pm_i$ for all $i \geq 1$ and the eigenvalues $1/2 > \lambda_1\geq \ldots \geq \lambda_i \rightarrow 0 $ as $i \rightarrow \infty$,
by using the self-adjointness and the compactness of the operator $\mathcal{K}_{\partial D}^*$ over
$\mathcal{H}$ and the Hilbert-Schmidt theorem; see \cite{shapiro}.
For notational sake, we will often write $\lambda_i^{\pm} := \pm \lambda_i$ in our subsequent discussions.
Then we can decompose the operator $\mathcal{K}_{\partial D}^*$ as
\beqn
\mathcal{K}_{\partial D}^* =
\sum_{i=1}^\infty \bigg\{ \lambda_i^+ \langle  \phi_i^+,\cdot \rangle_\mathcal{H}\, \,\phi_i^+ + \lambda_i^-  \langle  \phi_i^-,\cdot \rangle_\mathcal{H}\, \,\phi^-_i \bigg\}=\sum_{i=1}^\infty \lambda_i \bigg\{ \langle  \phi_i^+,\cdot \rangle_\mathcal{H}\, \,\phi_i^+ -  \langle  \phi_i^-,\cdot \rangle_\mathcal{H}\, \,\phi^-_i \bigg\} \,.
\eqn
Note that as $\mathcal{K}_{\partial D}^* $ is a pseudo-differential operator of order $-1$, the eigenfunctions $\phi^\pm_i$
oscillate as $1/\lambda_i$\,, and there exists a positive constant $C$ such that
$$
\frac{|| \frac{\partial \phi^\pm_i}{\partial T}||_{L^2(\partial D)}}{|| {\phi^\pm_i}||_{L^2(\partial D)}} \lesssim \frac{C}{\lambda_i},
$$
where $\partial/\partial T$ denotes the tangential derivative.

Now, given the Neumann-Poincar\'e operator $\mathcal{K}_{\partial D}^*$ corresponding to a shape $D$ ($D$ being an open domain
with $\mathcal{C}^{1,\alpha}$ boundary), we define, for any holomorphic function $f$ on an open domain $U \subset \mathbb{C}$ containing the spectrum of $\mathcal{K}_{\partial D}^*$, the following notion
\beqn
f (\mathcal{K}_{\partial D}^*) : = \sum_{i=1}^\infty
\bigg[
f (\lambda_i^+) \langle  \phi^+_i,\cdot \rangle_\mathcal{H} \,\phi^+_i + f (\lambda_i^-) \langle  \phi^-_i,\cdot \rangle_\mathcal{H} \,\phi^-_i \bigg]\, . \l{eq:fstar}
\eqn
Clearly, if $f$ is a polynomial in $z \in \mathbb{C}$, say $f(z) := \sum_{i=0}^N a_i z^i$ for some $N \in \mathbb{N}$,
the definition (\ref{eq:fstar})
coincides with the conventional one, i.e., $f (\mathcal{K}_{\partial D}^*) = \sum_{i=0}^N a_i \left(\mathcal{K}_{\partial D}^*\right)^i$, where $\left(\mathcal{K}_{\partial D}^*\right)^i$ means the composition of the operator $i$ times.  For our subsequent description, we may write for any $\phi\in L^2(\partial D)$ that
$$
\langle  \phi,y \rangle_{L^2(\partial D)}: = \int_{\partial D} y \phi(y) d \sigma(y)\,.
$$
Then we have the following representation result.
\begin{Lemma}
\label{lemmahol}
Given a shape $D$ and the corresponding Neumann-Poincar\'e operator $\mathcal{K}_{\partial D}^*$,
the following identity holds for the polarization tensor $M (\lambda,D)$ in \eqref{polarizationmatrix}
and any holomorphic function $f$ on an open domain $U \subset \mathbb{C}$
containing the spectrum $\sigma(\mathcal{K}_{\partial D}^*)$ of $\mathcal{K}_{\partial D}^*$:
\beqn
\f{1}{2 \pi i} \int_{\gamma} f(\lambda) M(\lambda,\partial D) d \lambda
=\int_{\partial D} y  f(\mathcal{K}_{\partial D}^*) [\nu](y) d \sigma(y)
= \sum_{i=1}^\infty \bigg[ c_i^+ f (\lambda_i^+) + c_i^- f (\lambda_i^-) \bigg] \,,
\label{Magic formula}
\eqn
where $\gamma$ is an arbitrary simple closed curve in $U$ enclosing $\sigma(\mathcal{K}_{\partial D}^*)$,
and $c_i^+ $ and $c_i^-$ are defined by
\beqn
c_i^+ := \langle  \nu, \phi^+_i \rangle_\mathcal{H} \langle  \phi^+_i,y \rangle_{L^2(\partial D)} \,,
\quad
c_i^- := \langle  \nu, \phi^-_i \rangle_\mathcal{H} \langle  \phi^-_i,y \rangle_{L^2(\partial D)}\,.
\eqn
\end{Lemma}
\begin{proof}
By the holomorphic functional calculus, we know for any holomorphic function $f$ on an open domain $U \subset \mathbb{C}$  containing $\sigma(\mathcal{K}_{\partial D}^*)$ and any simple closed curve $\gamma$ in $U$
enclosing $\sigma(\mathcal{K}_{\partial D}^*)$ that
\beqn
\f{1}{2 \pi i} \int_{\gamma} f(\lambda) R_{\lambda}(\mathcal{K}_{\partial D}^*) d \lambda = f (\mathcal{K}_{\partial D}^*)
= \sum_{i=1}^\infty
\bigg[
f (\lambda_i^+) \langle  \phi^+_i,\cdot \rangle_\mathcal{H} \,\phi^+_i + f (\lambda_i^-) \langle  \phi^-_i,\cdot \rangle_\mathcal{H} \,\phi^-_i \bigg] \,.
\label{fK}
\eqn
Combining this with (\ref{polarizationmatrix}), we readily derive that
\beqn
\f{1}{2 \pi i} \int_{\gamma} f(\lambda) M(\lambda,\partial D) d \lambda
=\int_{\partial D} y  f(\mathcal{K}_{\partial D}^*) [\nu](y) d \sigma(y)\,.
\eqn
Now the desired representation comes from the above two identities.
\end{proof}
We note that even if $\partial D$ is only Lipschitz, a similar result can be obtained for the (noncompact) operator $\mathcal{K}^*_{\partial D}$ from the spectral decomposition $\mathcal{K}_{\partial D}^* = \int \lambda d E_{\lambda}$ where $E_{\lambda}$ is the projection-valued measure. However, we will not pursue in this direction
for the sake of simplicity. We refer the reader to \cite{helsing}.

Based on the relation (\ref{Magic formula}), we can make use of
different choices of holomorphic functions $f$ to reconstruct the eigenvalues $\lambda_i$ of $\mathcal{K}^*_{\partial D}$
from its polarization tensor. One of the methods is based on the following observation.
For any $n \in \mathbb{N}$ we define
\beqn
h^{(n)}_1  &: = &\f{1}{2 \pi i} \int_{\gamma} \lambda^{2n} M(\lambda,\partial D) d \lambda \, ,
\label{hn1} \\
h^{(n)}_j  &: = &\f{1}{2 \pi i} \int_{\gamma} \lambda^{2n} M(\lambda,\partial D) d \lambda - \sum_{i = 1}^{j-1} \left( c_i^+ + c_i^- \right) \lambda_i^{2n} \quad \text{for } j >1 \,,
\label{hnj}
\eqn
then we come to the following corollary.
\begin{Corollary}
\label{method1}
Assume that all the eigenvalues of $\mathcal{K}_{\partial D}^*$ are simple.
Then for any $j \in \mathbb{N}$ such that $c_j^+ + c_j^-\ne 0$, it holds that
\beqn
\lim_{n \rightarrow \infty} \f{h^{(n)}_j}{h^{(n-1)}_j} = \lambda_j^2 \quad \text{ and } \quad \lim_{n \rightarrow \infty}  \f{h^{(n)}_j}{\lambda_j^{2n}} = c_i^+ + c_i^-  .
\label{obtainlambdaj}
\eqn
\end{Corollary}
\begin{proof}
Taking $f = \lambda^{2n}$ for $n \in \mathbb{N}$ in (\ref{Magic formula}), we have
\beqn
\f{1}{2 \pi i} \int_{\gamma} \lambda^{2n} M(\lambda,\partial D) d \lambda
= \sum_{i=1}^\infty \bigg[ c_i^+ (\lambda_i)^{2n} + c_i^- ( - \lambda_i)^{2n} \bigg]
= \sum_{i=1}^\infty \left( c_i^+ + c_i^- \right) \lambda_i^{2n}  \, .
\label{power}
\eqn
Substituting \eqref{power} into \eqref{hn1}-\eqref{hnj}, we get that for all $j \in \mathbb{N}$,
\beqn
h^{(n)}_j  = \sum_{i = j}^{\infty}\left( c_i^+ + c_i^- \right) \lambda_i^{2n} \, .
\label{hnjexp}
\eqn
Noting that all the eigenvalues of $\mathcal{K}_{\partial D}^*$ are simple and
$c_j^+ + c_j^-\ne 0$, we readily obtain from \eqref{hnjexp} that
\beqn
\lim_{n \rightarrow \infty} \f{h^{(n)}_j}{h^{(n-1)}_j} = \lim_{n \rightarrow \infty} \f{\sum_{i = j}^{\infty} \left( c_i^+ + c_i^- \right) \lambda_i^{2n}  }{\sum_{i = j}^{\infty} \left( c_i^+ + c_i^- \right) \lambda_i^{2(n-1)} } = \lambda_j^2\,,
\eqn
and
\beqn
\lim_{n \rightarrow \infty}  \f{h^{(n)}_j}{\lambda_j^{2n}} = \lim_{n \rightarrow \infty} \sum_{i = j}^{\infty}\left( c_i^+ + c_i^- \right) \left(\f{\lambda_i}{\lambda_j}\right)^{2n} = c_j^+ + c_j^- \, .
\eqn
This gives the conclusion of the corollary.
\end{proof}

With the help of Corollary \ref{method1}, we can propose the following method
to reconstruct the Fredholm eigenvalues from the polarization tensor at multiple contrasts.

\ms
\noindent \textbf{Method 1}.
Given two integers $J, N \in \mathbb{N}$. For $j = 1,2, \cdots,J$ and $n =
1,2,\cdots,N$, compute ${h^{(n)}_j}$ based on \eqref{hn1}-\eqref{hnj}, then
compute the square root of the quotient
\[
\sqrt{ {h^{(n)}_j}/{h^{(n-1)}_j} }
\]
for the approximation of the eigenvalue $\lambda_j$.

\ms
Next, we introduce another reconstruction method.
For a $\sigma_0>0$, $t \in \mathbb{R}$ and
a simple closed curve $\gamma$ enclosing $\sigma(\mathcal{K}_{\partial D}^*)$, we define
\begin{eqnarray}
\Phi_{\sigma_0,\gamma}(t) := \f{1}{2 \pi i} \int_{\gamma} \exp\left(-\f{(\lambda-t)^2}{2 \sigma_0^2}\right) M(\lambda,\partial D) d \lambda \, .
\label{defphi}
\end{eqnarray}
Then by taking a different holomorphic function $f$ in (\ref{Magic formula}),
we have the following useful result from Lemma \ref{lemmahol}.
\begin{Corollary}
\label{method2}
Given a shape $D$ and the corresponding Neumann-Poincar\'e operator $\mathcal{K}_{\partial D}^*$,
the following equality holds
\begin{eqnarray}
\Phi_{\sigma_0,\gamma}(t) =
\sum_{i=1}^\infty \left[ c_i^+ \exp\left(-\f{(\lambda_i-t)^2}{2 \sigma_0^2}\right) + c_i^- \exp\left(-\f{(\lambda_i+t)^2}{2 \sigma_0^2}\right) \right]  \, .
\label{bump}
\end{eqnarray}
\end{Corollary}
\begin{proof}
For $\sigma_0>0$ and $t \in \mathbb{R}$,
let $f(\lambda,t) := \exp (-\f{(\lambda-t)^2}{2 \sigma_0^2})$.  As $f(\lambda, t)$ is holomorphic with respect to $\lambda$,
we can substitute it in (\ref{Magic formula}) to get the desired representation.
\end{proof}

\ms
Noting that the function $\exp (-\f{(\lambda_i-t)^2}{2 \sigma_0^2})$
achieves its maximum
at $t=\lambda_i$ and decays exponentially away from $t=\lambda_i$, we
observe from \eqref{bump} that
 the local extrema of the function $\Phi_{\sigma_0,\gamma}(t)$ are
approximately located at the eigenvalues $\lambda_i$ of operator
$\mathcal{K}_{\partial D}^*$.
So we can reconstruct the eigenvalues $\lambda_i$ by evaluating the
local extrema of function $\Phi_{\sigma_0,\gamma}(t)$. This leads us to the following second reconstruction
method.

\ms
\noindent \textbf{Method 2}.
Given a small constant $\sigma_0>0$. Evaluate
function $\Phi_{\sigma_0,\gamma}(t)$ in (\ref{defphi}) for $t \in [-1/2,
1/2]$. Then locate the local extrema of function $\Phi_{\sigma_0,\gamma}(t)$ one by one,
starting from the one with the largest magnitude of $c_i^{\pm}$, then moving to the one
with the second largest magnitude of $c_i^{\pm}$, and so on.

\subsection{Numerical results} \label{numerical12}
In this subsection, we will first present some numerical results on the approximations of the Fredholm eigenvalues
and the decay properties of eigenvalues. Then we shall focus on the inverse problem to reconstruct the Fredholm eigenvalues from the observed PT at multiple contrasts.


For the approximations of the Neumann-Poincar\'e operator and Fredholm eigenvalues, we use a fine mesh of size
$h = 1/1024$ to discretize the integral operator (\ref{operatorK})
by the trapezoidal quadrature rule
over the curve $\partial D$, and compute
the eigenvalues of $\mathcal{K}^*_{\partial D}$.

For a given shape $D$, we plot the decay of eigenvalues and the growth of oscillation.
Let $\lambda_i$ be the $i$-th eigenvalue and $\phi_i^+$ be the corresponding eigenfunction.
Then we define the oscillation of the eigenfunction $\phi_i^+$ by
    \begin{equation}
    a_i := \frac{||\f{\partial \phi_i^+ }{\partial T} ||_{L^2(\p D)}}{||\phi^+_i||_{L^2(\p D)}} \, , \quad i \geq 1.
    \label{oscillation}
    \end{equation}

In Figures \ref{forwardnumerics_ellipse} and \ref{forwardnumerics_heart}, we can see
the detailed changes of $\lambda_i$ against $i$, $\log \lambda_i$ against $i$, and $a_i$
against $i$, from which one can observe the decay of eigenvalues and the growth of oscillation of eigenfunctions,
corresponding to two domains $D$, an ellipse of the form
     \begin{equation}
        \f{x^2}{4} +  y^2 = 1 \, , \quad x, y \in \mathbb{R}\,,
        \label{ellipsehaha}
    \end{equation}
and a heart-shaped domain of the form  (with $\delta = 0.8$ and $m = 1$):
    \begin{equation}
        r = 1+ \delta \sin(m\theta)\, , \quad \theta \in (0,2\pi]\,.
        \label{circle_perturb_form}
    \end{equation}

\begin{figurehere}
 \hfill{}\includegraphics[clip,width=0.44\textwidth]{./ellipse}\hfill{}
 \hfill{}\includegraphics[clip,width=0.44\textwidth]{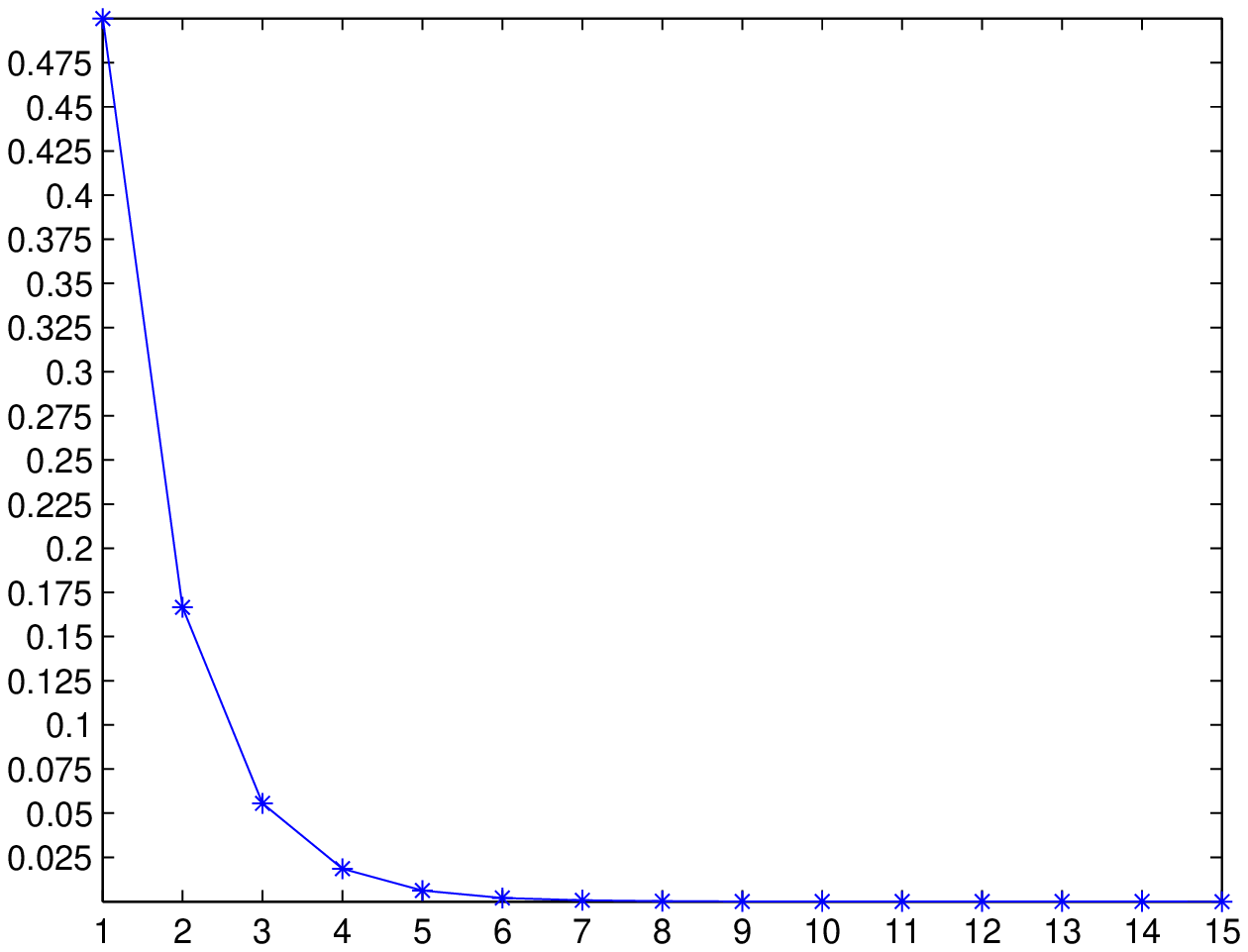}\hfill{}

 \hfill{}(a)\hfill{} \hfill{}(b)\hfill{}

\hfill{}\includegraphics[clip,width=0.44\textwidth]{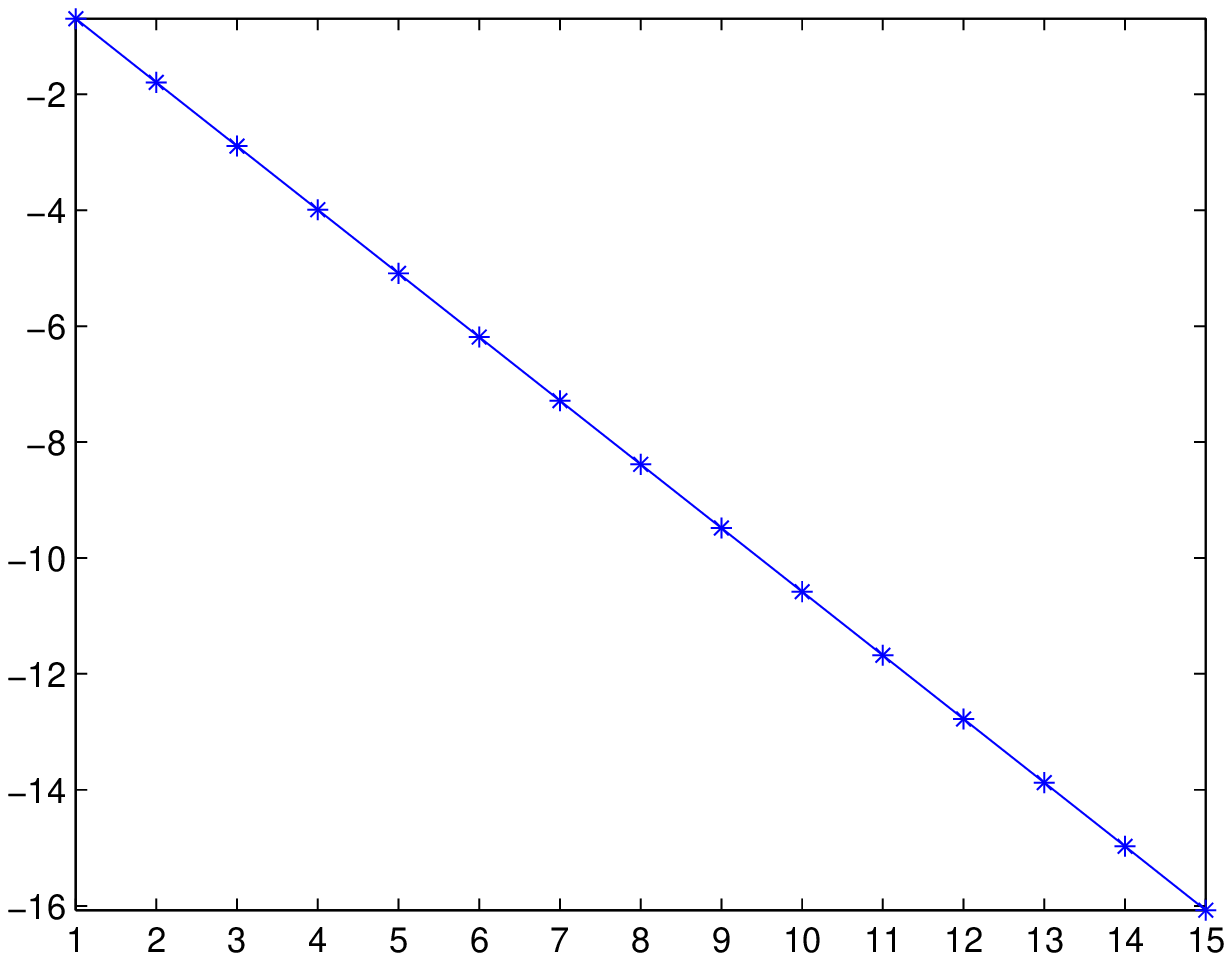}\hfill{}
 \hfill{}\includegraphics[clip,width=0.44\textwidth]{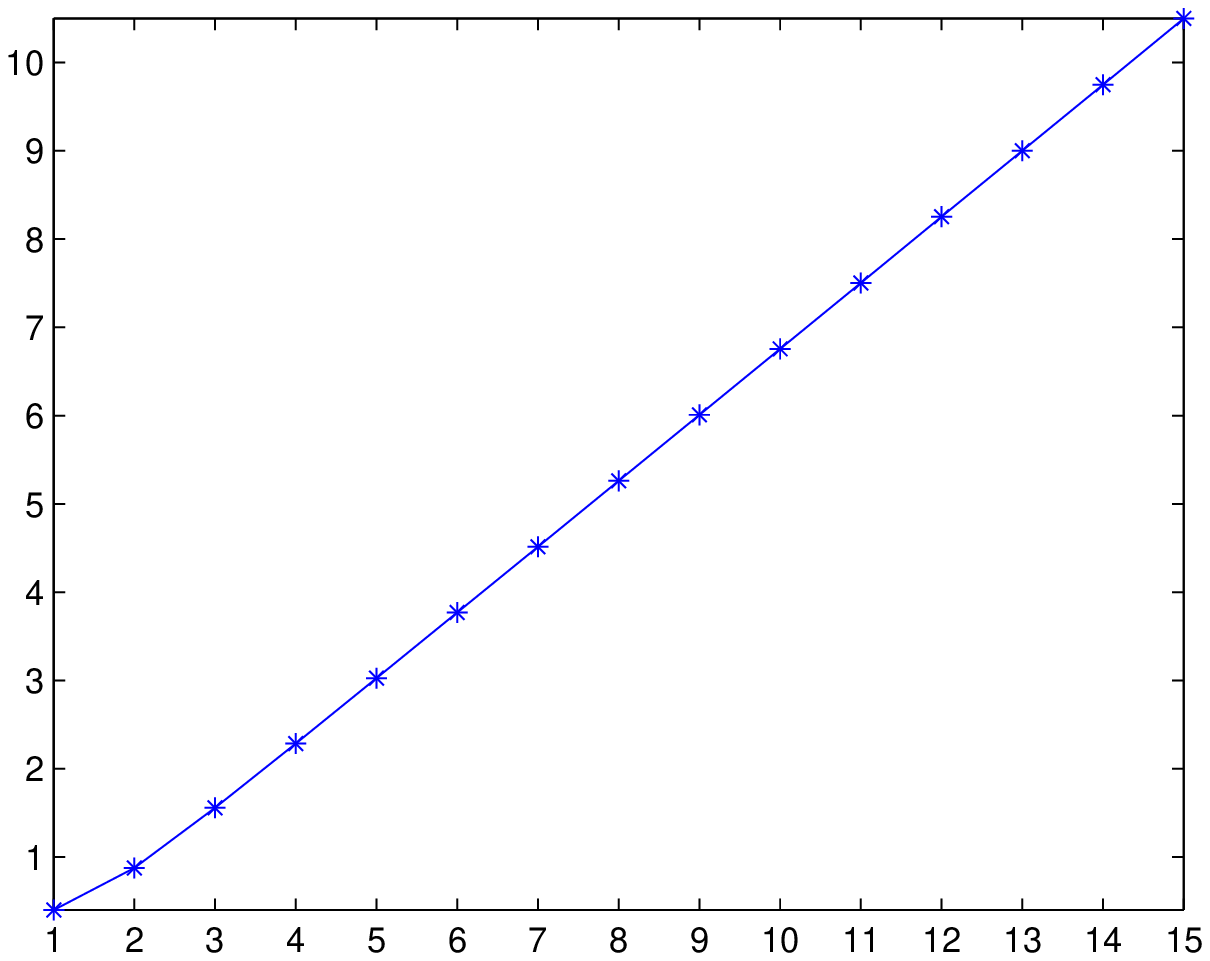}\hfill{}

 \hfill{}(c)\hfill{} \hfill{}(d)\hfill{}
  \vskip -0.15truecm
\caption{\small (a) Domain $D$; (b) $\lambda_i$ against $i$; (c) $\log \lambda_i$ against $i$; (d) $a_i$ against $i$.
     }\label{forwardnumerics_ellipse}
 \end{figurehere}

 \begin{figurehere}
 \hfill{}\includegraphics[clip,width=0.44\textwidth]{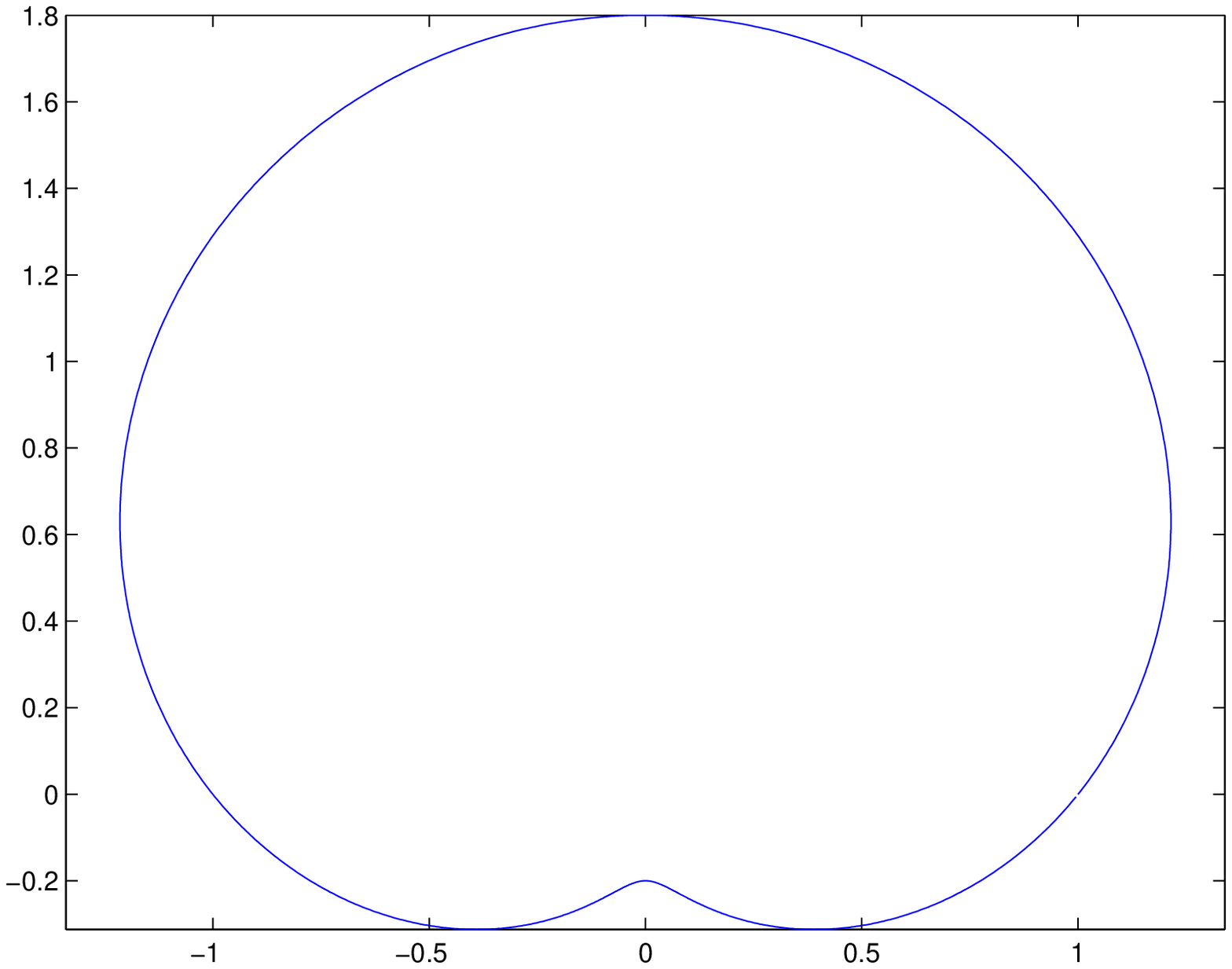}\hfill{}
 \hfill{}\includegraphics[clip,width=0.44\textwidth]{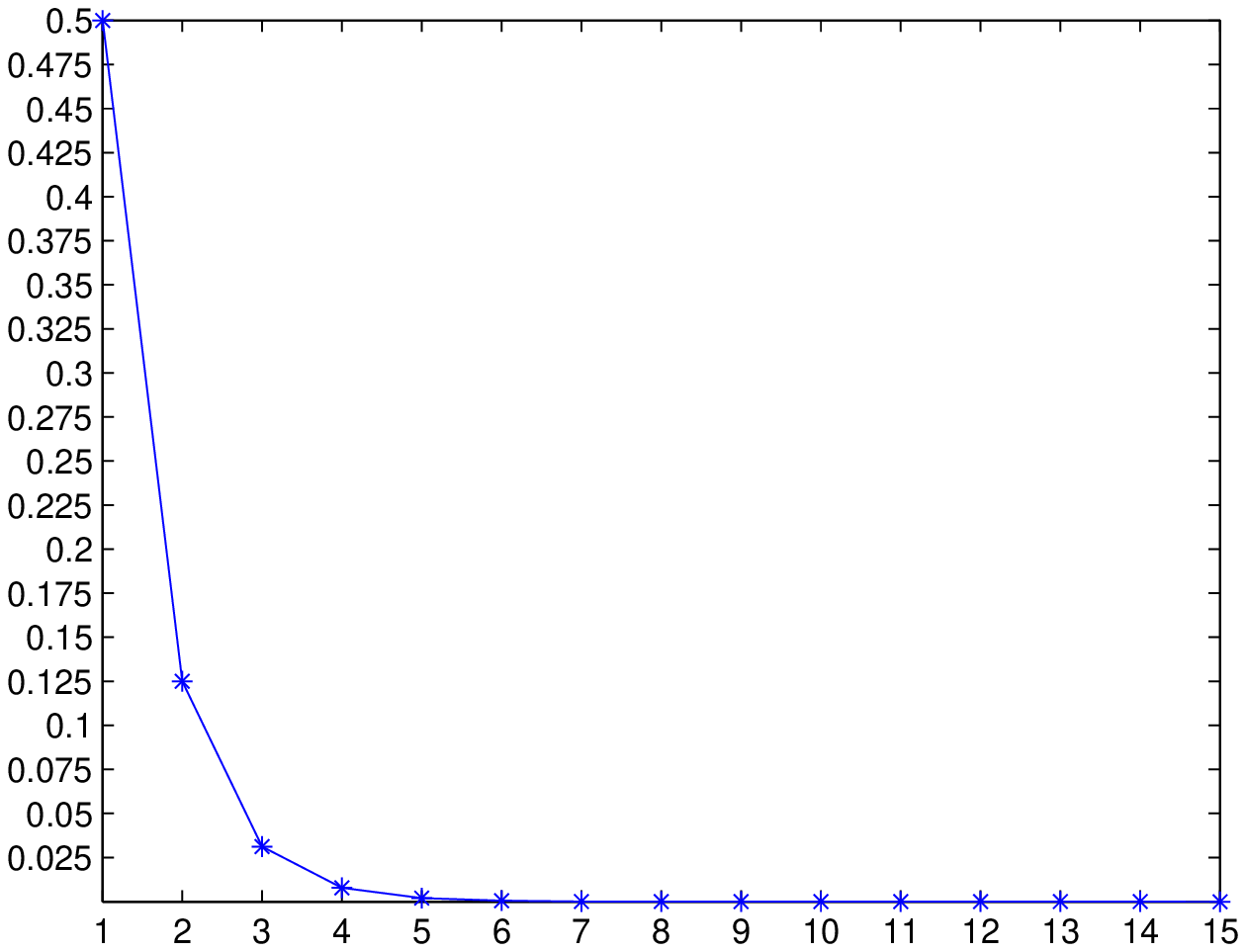}\hfill{}

 \hfill{}(a)\hfill{} \hfill{}(b)\hfill{}

\hfill{}\includegraphics[clip,width=0.44\textwidth]{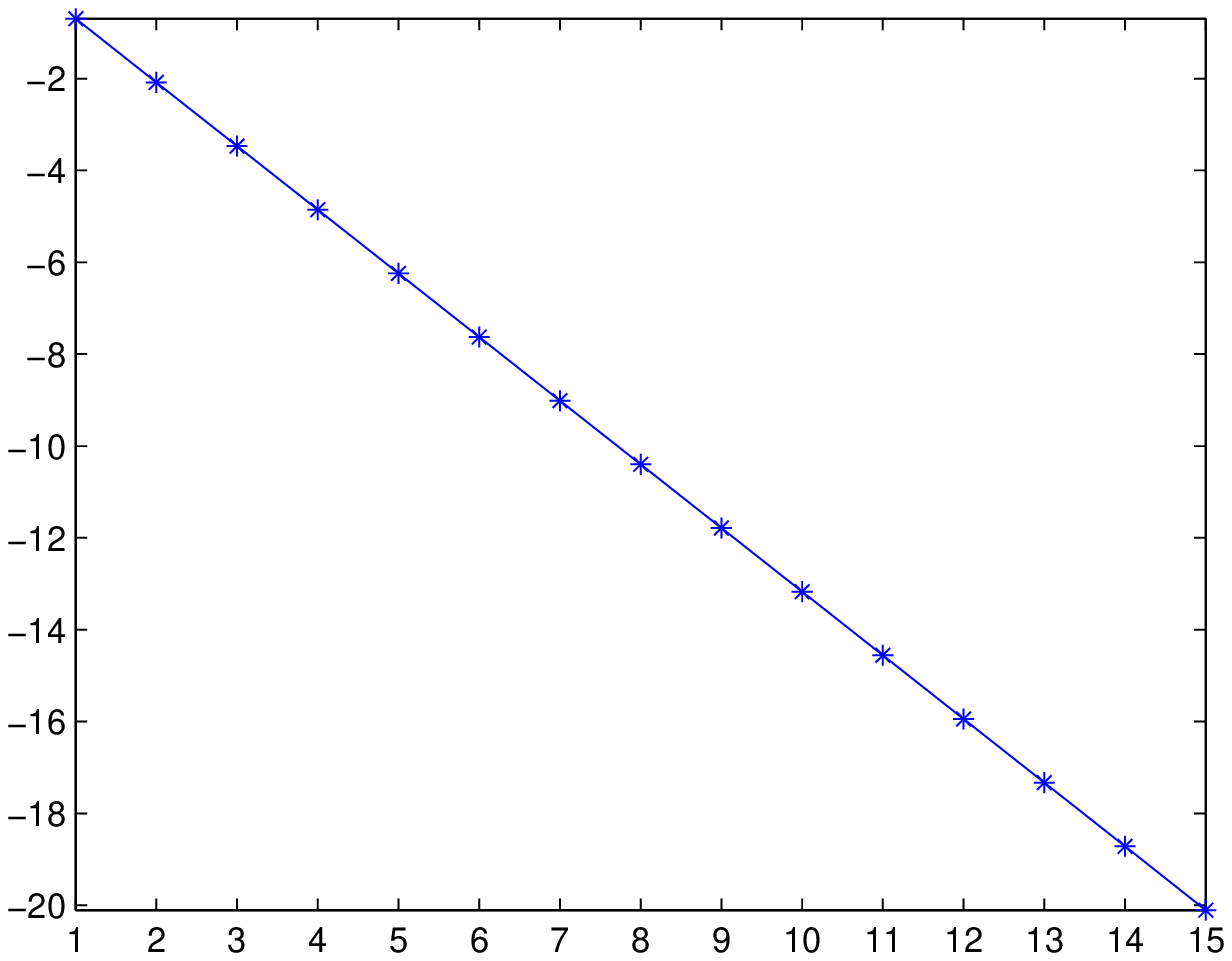}\hfill{}
 \hfill{}\includegraphics[clip,width=0.44\textwidth]{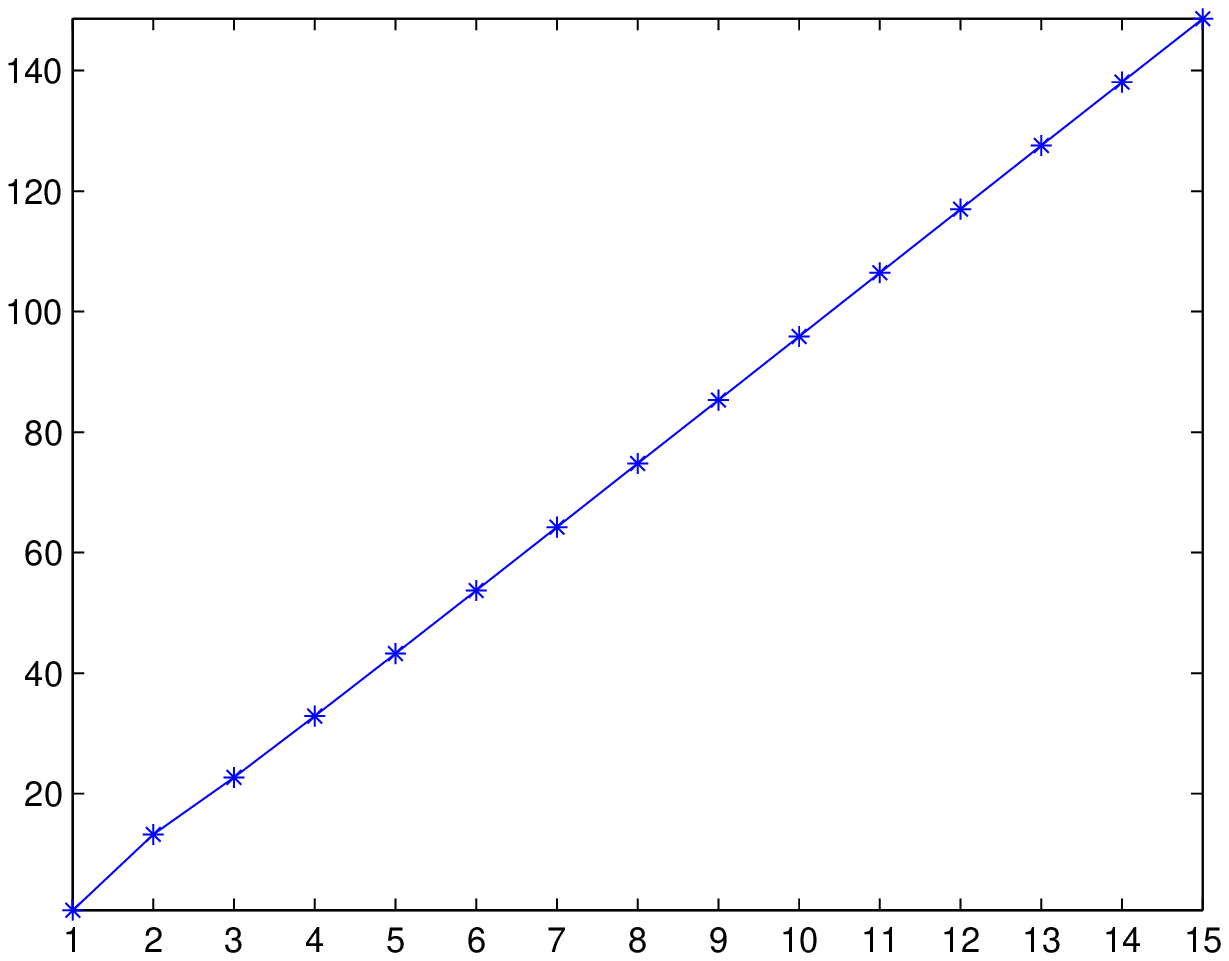}\hfill{}

 \hfill{}(c)\hfill{} \hfill{}(d)\hfill{}
  \vskip -0.15truecm
\caption{\small (a) Domain $D$; (b) $\lambda_i$ against $i$; (c) $\log \lambda_i$ against $i$; (d) $a_i$ against $i$.
     }\label{forwardnumerics_heart}
 \end{figurehere}

\ms
Next, we carry out some numerical examples for the reconstructions of Fredholm eigenvalues from PT
at multiple contrasts. The forward data is obtained by first approximating the Neumann-Poincar\'e operator as it was
done earlier in this subsection, then the PT, $M(\lambda, D)=(m_{ij}(\lambda,D))_{i,j=1}^d$,
is calculated based on (\ref{polarizationmatrix}) using the trapezoidal rule over a fine mesh of size $h = 1/1024$
on $\partial D$. Values of $M(\lambda, D)$ are obtained for $\lambda \in \mathbb{C}$ on the grid points of a uniform mesh of size $1 /100$
over the curve $\gamma$:
\beqn
\gamma = \{ 0.23\, e^{2 \pi i \theta} + 0.3 \,  \big| \,  0\le \theta\le 1\, \} \, , \l{eq:curve}
\eqn
and are regarded as the observed data for the reconstructions of the Fredholm eigenvalues.  For the numerical comparisons,
we have implemented both Methods 1 and 2 in Section \ref{func}.
We notice that, in Method 1, quadrature rules with accuracy of very higher orders
are necessary for the approximation of the contour integration
in order to accurately approximate $h^{(n)}_1$ in (\ref{hn1}) for large $n \in \mathbb{N}$,
which is the case for an accurate estimate of eigenvalues.
Hence Method 1 may be rather expensive, and we shall demonstrate
only the reconstructions by Method 2 below.


By considering only those eigenvalues lying inside $\gamma$ (which are all positive), we compute (\ref{defphi}) in
our implementations of Method 2 as follows:
{\small
\beqn
\Phi_{\sigma_0,\gamma}(t) := \f{1}{2 \pi i} \int_{\gamma} \exp\left(-\f{(\lambda-t)^2}{2 \sigma_0^2}\right) M(\lambda,\partial D) d \lambda
= \sum_{ 0.07 < \lambda_i < 0.53 } c_i^+ \exp\left(-\f{(\lambda_i-t)^2}{2 \sigma_0^2}\right) \,
\label{bump2}
\eqn
}with
$\sigma_0 = 0.05$.
Then we can locate the local extrema of function (\ref{bump2}) one by one, starting from the one with the largest
magnitude of $c_i^+$, then to the one with the second largest magnitude of $c_i^+$, and so on.
This process provides us with a set of approximate eigenvalues from the knowledge
of polarization tensor $M(\lambda,\partial D)$ over $\gamma$.
The exact eigenvalues and the approximate ones
obtained from the above described Method 2 are listed in
Table \ref{tab:tablekite} for the kite-shaped domain $D$ of the form
\beqn
  x = \cos \theta +0.65 \cos 2\theta -0.65\, , \quad  y = 1.5 \sin \theta  \, , \quad \theta \in (0,2\pi]\, ,
  \label{kiteformula}
\eqn
a pear-shaped domain in Table  \ref{tab:tablepear}
and a floriform domain with 3 petals in Table \ref{tab:table3floriform}. Here
the latter two domains are of the form (\ref{circle_perturb_form})
with the same parameter $m = 3$ but a different $\delta$, i.e., $\delta = 0.3$ and $0.6$.

     \begin{tablehere}
     \begin{center}
     \begin{tabular}{ | c | c | c | c |}
     \hline
     Eigenvalues & Exact solutions & Approximate solutions\\
     \hline
     First   &  0.5000   & 0.5000 \\
     \hline
     Second    &  0.2707    & 0.2700  \\
     \hline
     Third    &  0.1902    & 0.1800  \\
     \hline
     Fourth   &  0.0891    & 0.0900 \\
     \hline
     Fifth    &  0.0718    & 0.0700  \\
     \hline
     \end{tabular}
     \end{center}
     \vskip -0.3truecm
     \caption{\label{tab:tablekite}\emph{The first 5 reconstructed eigenvalues for the kite-shaped domain.}}
     \end{tablehere}

    \begin{tablehere}
     \begin{center}
     \begin{tabular}{ | c | c | c | c |}
     \hline
     Eigenvalues & Exact solutions & Approximate solutions\\
     \hline
     First   &  0.5000   & 0.5000 \\
     \hline
     Second    &  0.1035    & 0.1050  \\
     \hline
     Third    &  0.1035    & 0.1050  \\
     \hline
     \end{tabular}
     \end{center}
      \vskip -0.3truecm
     \caption{\label{tab:tablepear}\emph{The first 3 reconstructed eigenvalues for the pear-shaped domain.}}
     \end{tablehere}

     \begin{tablehere}
     \begin{center}
     \begin{tabular}{ | c | c | c | c |}
     \hline
     Eigenvalues & Exact solutions & Approximate solutions\\
     \hline
     First   &  0.5000   & 0.5000 \\
     \hline
     Second    &  0.3322       & 0.3300  \\
     \hline
     Third    &  0.3322    & 0.3300  \\
     \hline
     Fourth    &  0.1404    & 0.1300  \\
     \hline
      fifth    &  0.1404    & 0.1300  \\
     \hline
     \end{tabular}
     \end{center}
      \vskip -0.3truecm
     \caption{\label{tab:table3floriform}\emph{The first 5 reconstructed eigenvalues for the floriform domain.}}
     \end{tablehere}

As we can observe from Tables \ref{tab:tablekite},  \ref{tab:tablepear}
and \ref{tab:table3floriform}, the reconstructed eigenvalues are rather satisfactory and accurate
in view of the severe ill-posedness of recovering eigenvalues.

\section{Hadamard's formula for the Fredholm eigenvalues} \label{sec4}
In this section, we turn our attention to the optimal shape design problem given the Fredholm eigenvalues corresponding
to a geometric shape.
Our tactic to approach the problem is via an optimization of a least-squares functional.
For this purpose,
we first discuss how to obtain the shape derivatives of the Neumann-Poincar\'e operator
and the Fredholm eigenvalues. These derivatives
are needed when we compute the gradient of least-squares functional concerned.

To start with, we first focus on the shape derivative of the Neumann-Poincar\'e operator, which was derived in \cite{zribi}.
We need some new notations for the description of derivative.
Given a shape $D$ and $a,b \in \mathbb{R}$ with $a < b$, we consider an arc-length parametrization of $\partial D$,
$
X(t): [a,b] \rightarrow \partial D\,.
$
Let $T(x)$ and $\nu(x)$ be respectively the tangent vector and the outward unit normal
to $\partial D$ at $x\in \p D$, and $\tau(x)$ be  the curvature defined by
\beqn
X''(t) = \tau(x) \nu (x) \, .
\eqn
Now consider an $\varepsilon$-perturbation of $D$, namely
$\partial D_\varepsilon$ is given by
\beqn
    \partial D_\varepsilon :=\{\widetilde{x} \, \big| \, \widetilde{x} = x+ \varepsilon h(x) \nu(x) \, , \, x \in \partial D \}\,,
    \label{variationvariation}
\eqn
where $h \in \mathcal{C}^1(\partial D)$.
For two arbitrary points $x,y \in \partial D$ such that $x = X(t), y = X(s)$ for some $t, s \in [a,b]$,
we define
\beqn
F_h(x,y) = \f{\langle  x-y,h(t)x(x) - h(s)\nu(y) \rangle  }{|x-y|^2} \; \text{ and }  \; G_h(x,y) = \f{|h(t)x(x) - h(s)\nu(y)|^2}{|x-y|^2} \, ,
\label{workint0}
\eqn
and $F_{h,n}$ as the coefficients in the following series
\beqn
    \sum_{n=0}^\infty \varepsilon^n F_{h,n}(x,y) := \f{1}{1 + 2 \varepsilon F_h(x,y) + \varepsilon^2 G_h(x,y)} \f{\sqrt{(1 -\varepsilon \tau(y)h(s))^2 + \varepsilon^2 (h'(s))^2}}{\sqrt{(1 -\varepsilon \tau(x)h(t))^2 + \varepsilon^2 (h'(t))^2}}\,,
    \label{seriesint1}
\eqn
where the series converges absolutely and uniformly \cite{zribi}.
We can directly see that
\beqn
F_{h,0}(x,y) = 0 \, \quad \text{ and }  \quad F_{h,1}(x,y) = -2 F(x,y) + \tau(x)h(x)-\tau(y)h(y) \, .
\eqn
For any two points $\widetilde{x}, \widetilde{y} \in \partial D_\varepsilon$
such that $\widetilde{x} := x + \varepsilon h(t) \nu(x)$ and $\widetilde{y} := y + \varepsilon h(s) \nu(y)$,
we write $\mathbb{K}_{h,n}$ as the coefficients in the following series
\beqn
\sum_{n=0}^\infty \varepsilon^n \mathbb{K}_{h,n}(x,y) d \sigma(y): = \f{\langle \widetilde{x}-\widetilde{y}, \widetilde{\nu}(\widetilde{x}) \rangle  }{|\widetilde{x}-\widetilde{y}|^2} d \sigma_\varepsilon(\widetilde{y}) \, ,
\label{workint}
\eqn
where $\widetilde{\nu}(\widetilde{x})$ denotes the outward unit normal to $\partial D_\varepsilon$ at $\widetilde{x}$, while $d \sigma(y)$ and $d \sigma_\varepsilon(\widetilde{y})$ are the length elements on $\partial D$ at $y$ and on $\partial D_\varepsilon$ at $\widetilde{y}$ respectively.  Then, following the argument in \cite{zribi} and using
(\ref{workint0})-(\ref{seriesint1}) and (\ref{workint}) we have
\beqnx
    \mathbb{K}_{h,0} = \f{\langle x-y, \nu(x) \rangle  }{|x-y|^2}, ~~\mathbb{K}_{h,1} = K_{h,0} F_{h,1} + K_{h,1}, ~~\mathbb{K}_{h,n} = F_{h,n} K_{h,0} + F_{h,n-1} K_{h,1} + F_{h,n-2} K_{h,2}\,,
\eqnx
for $n \geq 2$, where $K_{h,0}$, $K_{h,1}$ and $K_{h,2}$ are given by
\begin{eqnarray*}
 K_{h,0} &= &\f{\langle x-y, \nu(x) \rangle  }{|x-y|^2},\\
 K_{h,1} &=& \f{\langle h(t)\nu(x)-h(s)\nu(y), \nu(x) \rangle  }{|x-y|^2}- \f{\langle x-y ,\tau(x) h(t)\nu(x)+h'(t)T(x) \rangle  }{|x-y|^2},\\
 K_{h,2} &=& \f{\langle h(t)\nu(x)-h(s)\nu(y), \tau(x) h(t)\nu(x)-h'(t)T(x) \rangle  }{|x-y|^2}.
\end{eqnarray*}
Define a sequence of integral operators $\mathcal{K}^{(n)}_{D,h}$: $L^2 (\partial D) \to L^2 (\partial D)$ by
\beqn
   \mathcal{K}^{(n)}_{D,h} \phi(x) := \int_{\partial D} \mathbb{K}_{h,n}(x,y)\phi(y)d \sigma(y) ~~\forall\, \phi \in L^2 (\partial D)\,
   \label{kernalkernal}
\eqn
for $n \geq 0$. Let $\Psi_\varepsilon$ to be the diffeomorphism from $\partial D$ to $\partial D_{\varepsilon}$ given by $\Psi_\varepsilon(x) = x + \varepsilon h(t)\nu(x)$, then we have the following result from \cite{zribi}.
\begin{Lemma}
For $N \in \mathbb{N}$, there exists constant $C$ depending only on $N$, $||X||_{\mathcal{C}^2}$ and $||h||_{\mathcal{C}^1}$ such that
the following estimate holds for any $\widetilde{\phi} \in L^2 (\partial D_\varepsilon)$ and
$\phi := \tilde{\phi}\circ \Psi_\varepsilon$,
\beqn
   \bigg|\bigg| \mathcal{K}^*_{\partial D_\varepsilon}[ \tilde{\phi} ] \circ \Psi_\varepsilon - \mathcal{K}^*_{\partial D}[{\phi} ] - \sum_{n=1}^N \varepsilon^n  \mathcal{K}^{(n)}_{D,h} [\phi] \bigg|\bigg|_{L^2(\partial D)}
   \leq C \varepsilon^{N+1} ||\phi||_{L^2(\partial D)}\,.
   \label{seriesvariation}
\eqn
\end{Lemma}

From (\ref{seriesvariation}) we know the shape derivative of the Neumann-Poincar\'e operator at the variation $h$ and $D$\,:
\beqn
[\mathcal{D} (\mathcal{K}_{\partial D}^*) ](h) =  \mathcal{K}^{(1)}_{D,h} \,.
\label{D1}
\eqn

Next we turn our attention to the shape derivatives of the Fredholm eigenvalues.
By the Osborn's theorem \cite{osborn}, we have
\beqn
    | \lambda_i^{\pm}(D) - \lambda_i^{\pm} ({D_\varepsilon}) - \langle (\mathcal{K}^*_{\partial D} -  \mathcal{K}^*_{\partial D_\varepsilon} \circ \Psi_\varepsilon ) \phi_i^{\pm}(D), \phi_i^{\pm}(D)  \rangle   | \leq C ||\mathcal{K}^*_{\partial D} -  \mathcal{K}^*_{\partial D_\varepsilon} \circ \Psi_\varepsilon ||^2 \, ,
    \label{osborn}
\eqn
using the facts that $\mathcal{K}^*_{\partial D_\varepsilon}$ is collectively compact, i.e., $\{\mathcal{K}^*_{\partial D_\varepsilon}[\phi] : ||\phi||_{\mathcal{H}} \leq 1, \varepsilon \geq 0 \}$ is sequentially compact, and that $\mathcal{K}^*_{\partial D_\varepsilon} \rightarrow \mathcal{K}^*_{\partial D}$ pointwise.
Here $\phi_i^{\pm}(D)$ are the orthonormal eigenfunctions of the operator $ \mathcal{K}^*_{\partial D}$\,.
Now we can easily see from
(\ref{seriesvariation}) and (\ref{osborn})
the following estimates for the variation of eigenvalues:
\beqn
    | \lambda_i^{\pm}(D) - \lambda_i^{\pm} ({D_\varepsilon}) - \varepsilon \langle \mathcal{K}^{(1)}_{D,h}  \phi_i^{\pm}, \phi_i^{\pm}  \rangle   | \leq C \, \varepsilon^2.
\eqn
This yields the following result.
\begin{Proposition}
Let $\lambda_i^{\pm}$ be simple Fredholm eigenvalues, then their shape derivatives are given by
\beqn
[\mathcal{D} (\lambda_i^{\pm}) ]\big|_{D}(h) =  \langle \mathcal{K}^{(1)}_{D,h}  \phi_i^{\pm}, \phi_i^{\pm} \rangle   \,.
\label{D2}
\eqn
\end{Proposition}

It is worth mentioning that if $\lambda_i$ is a multiple eigenvalue, it may evolve, under perturbations, into
several separated and distinct eigenvalues.
The  splitting of eigenvalues may only become apparent at high orders in their Taylor expansions
with respect to the perturbation parameter.
The splitting  problem in the evaluation of perturbations in $\lambda_i$ can be addressed
using the arguments in \cite[Section 3.4]{AKL09}.

\section{Optimal shape design using partial spectral data} \label{shape}

\subsection{Shape design via optimization}

In this subsection, we formulate our design problem via an optimization framework. We first recall our shape design problem: given a set of eigenvalues $\{\pm \lambda_i (B)\}_{i=1}^N$ corresponding to a target shape $B$,
we intend to find a shape $D$ such that the eigenvalues of the Neumann-Poincar\'e operator $\mathcal{K}_{\partial D}^*$, denoted as $\lambda_i(D)$, are approximately equal to $\lambda_i(B)$, i.e., $ \lambda_i(D) \approx \lambda_i(B)$ for $1\le i\le N$.  In order to achieve this, we have to
introduce an appropriate objective functional. In view of the invariance of eigenvalues under rigid transformations and scaling,
some effective penalty and regularization terms should be incorporated in the functional to ensure the local existence
and uniqueness of the minimizers.
This leads us to the following nonlinear functional for our shape design
\beqn
    {\cal J}_{I,\alpha, \beta} (D)
    &=& \f{1}{2} \sum_{i=1}^I w_i^2 |\lambda_i (D) - \lambda_i (B)|^2  +  \f{\alpha}{2} (|D|-1)^2 + \f{\beta}{2} \left( \int_D 2 |x_1|^2 + |x_2|^2 \right)\label{functional2} \notag \\
    &:=&  ({\cal J}_I)_0 (D) +  \f{\alpha}{2} {\cal A}(D) + \f{\beta}{2} {\cal B}(D)\,, \label{functional}
\eqn
where $I\le N$ is a given integer index, $\alpha, \beta \in \mathbb{R}^+$
are the parameters for the penalty and the regularization respectively.
In view of the large variation of the magnitudes of eigenvalues, we have also introduced
some weights $w_i$ in \eqref{functional}, which we will naturally choose to be
$w_i = {1}/{\lambda_i(B)}$.



For most existing optimization algorithms, we need to compute the variational derivatives
of the functionals involved. For this purpose, we introduce some auxiliary tools.
\begin{Lemma}
For a given shape $D$ and $f \in L^1(D)$, the shape derivative of the integral
\beqn
I(D):=\int_{D} f dx
\eqn
is given by
\beqn
    [\mathcal{D} I ]|_{D}(h) = \int_{\partial D} f h dt \,
    \label{variationsolution}
\eqn
at the variation $h \in \mathcal{C}^1(\partial D)$.
\end{Lemma}

\begin{proof}
Given a shape $D$ and $a,b \in \mathbb{R}$ with $a < b$, let
$
X(t): [a,b] \rightarrow \partial D
$ be an arc-length parametrization of $\partial D$ and $\nu(x)$ be the outward unit normal to $\partial D$ at $x\in \p D$.
For a $h \in \mathcal{C}^1(\partial D)$ which is non-zero everywhere, we consider the $\varepsilon$-perturbation $D_\varepsilon$
of $D$
as in \eqref{variationvariation}. For sufficiently small $\varepsilon >0$, we consider a change of variables $(x_1,x_2) \mapsto
(\tilde \varepsilon,t)$ in an $\varepsilon$-tabular neighborhood of $\partial D$. Denoting $\det(u,v) = \det \begin{pmatrix} u_1 & v_1 \\ u_2 & v_2 \end{pmatrix}$ for any $u,v \in \mathbb{R}^2$, then we can write
\beqn
    \int_{D_\varepsilon} f dx = \int^{\varepsilon}_{0} \int_{\partial D} f \det(X'+\tilde\varepsilon h \nu', h \nu ) dt d \tilde\varepsilon + \int_{D} f dx \, . \label{labelint}
\eqn
Using the fact that $X' \bot \nu$ and $\nu' \bot \nu$, we can evaluate the shape derivative of $I$ at $h$ by
\beqn
     [\mathcal{D} I ]|_{D}(h) &=& \f{\partial}{\partial \varepsilon} \int_{D_\varepsilon} f dx \bigg|_{\varepsilon = 0} \notag \\
    &=&  \int_{\partial D} f \det(X'+\varepsilon h \nu', h \nu ) dt |_{\varepsilon = 0}  \notag \\
    &=&  \int_{\partial D} f \det(X' , h \nu ) dt  \notag \\
    &=&  \int_{\partial D} f h dt \, . \notag
    \label{intintint}
\eqn
This leads to the desired formula.
\end{proof}

Using \eqref{variationsolution}, we readily know the shape derivative of the following integrals at $h$:
\beqn
    [\mathcal{D} ( |D|)] (h) = \int_{\partial D} h dt \, , \quad \bigg[\mathcal{D} (\int_{D} 2 |x_1|^2 + |x_2|^2  dx )\bigg] (h) = \int_{\partial D} (2 |x_1|^2 + |x_2|^2) h dt \, .
    \label{D3}
\eqn

With the above preparations we can now discuss the  minimization
of functional \eqref{functional}.  In this work we will focus on the Gauss-Newton method for this minimization.
We first introduce some more notions. 

Given a shape $D$ and $I \in \mathbb{N}$,  we write the vector
$\lambda_{w}(D) = ( w_1 \lambda_1(D), \cdots, w_I \lambda_I(D) )^T$ with the superscript $T$ denoting the transpose, and
define the Jacobian of the map $D \mapsto \lambda_{w}(D)$:
\beqn
 \mathcal{J}_I|_{D} : L^2 (\partial D) &\mapsto& \mathbb{R}^I \notag \\
 \mathcal{J}_I|_{D}(h) &=& ( w_1[\mathcal{D} \lambda_1]\big |_{D} (h) , \cdots, w_I [\mathcal{D} \lambda_I]\big |_{D} (h) )^T,
\eqn
where $[\mathcal{D} \lambda_i] \big|_{D}(h)$ is the shape derivative of the Fredholm eigenvalue $\lambda_i(D)$
at $h$, which can be computed by \eqref{D2}.
Let $\mathcal{J}_I\big|_{D}^*$ and $[\mathcal{D} (|D|)]\big |_{D}^*$ be the respective $L^2(\partial D)$ adjoint of $\mathcal{J}_I\big|_{D}$ and $[\mathcal{D} (|D|)] \big|_{D}$, then the Gauss-Newton direction of \eqref{functional} for a shape $D$ can be written as
\beqn
N_{I,\alpha, \beta} (D)
    &:=&([\mathcal{J}_I \big|_{D}]^* [\mathcal{J}_I \big|_{D}] )^{-1} \mathcal{J}_I \big|_{D}^*  [ \lambda_w (D) - \lambda_w (B) ]  \notag\\
    &&+ \alpha [\mathcal{D} (|D|)] \big|_{D}^*  (|D|-1)  +  \beta\bigg[\mathcal{D}(\int_D 2|x_1|^2 + |x_2|^2)\bigg]\bigg|_{D}\,.
    \label{gaussnewtondirection}
\eqn
Now we are ready to formulate the Gauss-Newton method for the minimization
of functional \eqref{functional}: Let $D_n$ be the $n$-th approximation of the shape $D$, and
$X_n$ be its arc-length parametrization, then we update $X_n$  by the following iteration:
\beqn
    X_{n+1}= X_{n} - \gamma_n N_{I,\alpha_n, \beta_n} (D_{n})\,, \label{newton_formula}
\eqn
where $\gamma_n$, 
$\alpha_n$ and $\beta_n$ are parameters chosen at each iteration and $N_{I,\alpha_n, \beta_n} (D_{n})$
is the Gauss-Newton direction as defined in \eqref{gaussnewtondirection} with $\alpha = \alpha_n$ and $\beta=\beta_n$.
The choice of parameters $\gamma_n$, $\alpha_n$ and $\beta_n$ will be discussed in details
in the next subsection.

\subsection{Successive refinement for optimization and parameter selection}\l{sec:parameters}

In this subsection, we describe several detailed strategies for the minimization of the functional $\mathcal{J}_{I,\alpha, \beta}$ in (\ref{functional}).  These strategies are crucial for our algorithm to work efficiently, and to overcome the difficulties
arising from the strong nonlinearity and severe ill-posedness of the current shape design problem.

Our first strategy is a successive refinement technique for the minimization. 
This strategy is motivated by our observations from numerical experiments.
Due to the strong nonlinearity and ill-posedness,
iteration (\ref{newton_formula}) may stop at some local minima of (\ref{functional})
when $I \in \mathbb{N}$ is large. On the other hand, for small $I$, we observe that
iteration (\ref{newton_formula}) converges
often to a global minimum of (\ref{functional}) rapidly even with a poor initial guess.
But functional (\ref{functional}) does not capture fine features of the target shape
if $I$ is too small.  These observations motivate us with the following successive refinement strategy:
We first minimize $\mathcal{J}_{I,\alpha, \beta}$ in (\ref{functional}) with $I = 2$, then
minimize $\mathcal{J}_{I,\alpha, \beta}$ for $I = 3,\cdots,N$ recursively by
using the minimizer of $\mathcal{J}_{I-1,\alpha, \beta}$ as an initial guess.
As we will see in our numerical experiments,
this strategy works very effectively in avoiding the trapping of the minimization process
at some local minima as well as providing us with more fine details for our shape design.

The next strategy is on the choice of parameters $\alpha_n$ and $\beta_n$ for iteration (\ref{newton_formula}).
$\alpha_n$ and $\beta_n$ should be chosen such that the contributions on the search directions
in (\ref{newton_formula}) from three parts $(\mathcal{J}_I)_0 (D)$, $\mathcal{A}(D)$ and $\mathcal{B}(D)$ in (\ref{functional2}) are balanced at each iteration.
Under these considerations, a possible choice is that we first fix two small positive constants $C_1$ and $C_2 $,
then update $\alpha_n$ and $\beta_n$  at each iteration by
\beqn
\alpha_n = C_1 \f{(\mathcal{J}_{I})_0(D_n)}{\mathcal{A}(D_n)} \, , \quad \beta_n = C_1 \f{(\mathcal{J}_{I})_0(D_n)}{\mathcal{B}(D_n)} \, .
\label{choiceparameters}
\eqn

Our last strategy is on the choice of step size $\gamma_n$ along the Gauss-Newton direction
$N_{I,\alpha_n, \beta_n} (D_{n})$,
for which we will carry out the line search, namely
\beqn
\gamma_n = argmin \left\{\mathcal{J}_{I,\alpha_n, \beta_n} (X_{n} - \gamma
N_{I,\alpha_n, \beta_n} (D_{n})) : \gamma \in \mathbb{R}^+\right\} \,.
\label{choicegamma}
\eqn

Combining the above three strategies, we arrive at the successive refinement Gauss-Newton shape design algorithm.

\begin{center}
\textbf{Reconstruction Algorithm}
\end{center}
\begin{enumerate}
\item[\textbf{Step 1}]
Given a tolerance $\varepsilon$ and  an initial guess $D_{1,0}$.
\item[\textbf{Step 2}]
For $I = 1$ to $N$,
\begin{itemize}
\item[\textbf{Step 2.1}]
Set $n := 1$;
\item[\textbf{Step 2.2}]
Compute $\alpha_{I,n}$, $\beta_{I,n}$ as in (\ref{choiceparameters});
\item[\textbf{Step 2.3}]
Compute the Gauss-Newton direction $N_{I,\alpha_n, \beta_n} (D_{I,n})$ as in \eqref{gaussnewtondirection};

Find the step size $\gamma_n$ as in \eqref{choicegamma};  Then update $X_{I,n}$ by
\[
    X_{I,n+1} = X_{I,n} - \gamma_{I,n} N_{I,\alpha_n, \beta_n} (D_{I,n})\,;
\]
\item[\textbf{Step 2.4}]
If $| \mathcal{J}_{I,\alpha_n, \beta_n} (X_{I,n}) - \mathcal{J}_{I,\alpha_n, \beta_n} (X_{I,n+1}) | < \varepsilon $,
set $D_{I,\text{stab}}: = D_{I,n+1}$; otherwise set $n := n+1$ and go to \textbf{Step 2.2};
\item[\textbf{Step 2.5}]
Take $D_{I+1,0} := D_{I,\text{stab}}$.
\end{itemize}
\item[\textbf{Step 3}]
Find $n_0 \in \{1,2,..,N\}$ such that $D_{n_0,\text{stab}}$ has the minimal residue:
\[(\mathcal{J}_N)_0 (D_{n_0,\text{stab}}) = \min \limits_{I\in \{1,2,..,N\}} \big\{ (\mathcal{J}_N)_0 (D_{I,\text{stab}}) \big\} \, .\]
Output $D_{n_0,\text{stab}}$ and stop.
\end{enumerate}

\subsection{Numerical results} \label{numerical1}

In this section, we shall present several numerical examples to check the performance
of the newly proposed reconstruction algorithm  in section\,\ref{sec:parameters}
for the optimal shape design using partial spectral data.

Given a domain $D$, we first obtain the observed data of the forward problem, the Fredholm eigenvalues of $D$, as in section \ref{numerical12}. In order to test the robustness of our reconstruction algorithm, we introduce some
multiplicative random noise in the eigenvalues of the forward problem as follows:
\beqn
\lambda_i^{\sigma} = \lambda_i (1 + \sigma\, \xi )\,, \quad i = 1,\cdots,N \,,
\label{noise}
\eqn
where $\xi$ is uniformly distributed between -1 and 1 and $\sigma$ corresponds to the level of the noise in the data, which
is always set to be 1\% in all our examples.  It is well-known that the perturbations in the eigenvalues often affect the resulting
computations greatly in many applications. It is the same in our current cases.
When a new set of random noise is added in the eigenvalues as in (\ref{noise}), it gives us
a different set of observed data $\{\lambda_i^{\sigma}\}$. But it is interesting to us that
for each example we demonstrate in this section, we obtain only about 2 or 3 basic shapes by our reconstruction algorithm,
and all the other shapes obtained with different set of random noise are basically of very small perturbations
around these 2 or 3 basic shapes.
In our choices of $\alpha_n$, $\beta_n$ and tolerance $\varepsilon$, we take $C_1=C_2 =0.01$ in (\ref{choiceparameters}),
and $\varepsilon=5\times 10^{-4}$. And we will take the first 7 eigenvalues in the observed data,
namely  $N = 7$  in our reconstructions.

\ss
\no\textbf{Example 1}.
This example tests an ellipse of the form (\ref{ellipsehaha}) as the target shape;
see Figure \ref{invellipse}(a).
Figures \ref{invellipse}(c) and \ref{invellipse}(d) show two reconstructed shapes that appear most frequently with
different sets of random noise. The initial guess in the reconstruction is a shape of the form (\ref{circle_perturb_form})
with $\delta = 0.6$, $m = 5$; see Figure \ref{invellipse}(b). Clearly this is a very poor initial shape, but the reconstructed
shapes seem quite satisfactory.

\begin{figurehere}
 \hfill{}\includegraphics[clip,width=0.45\textwidth]{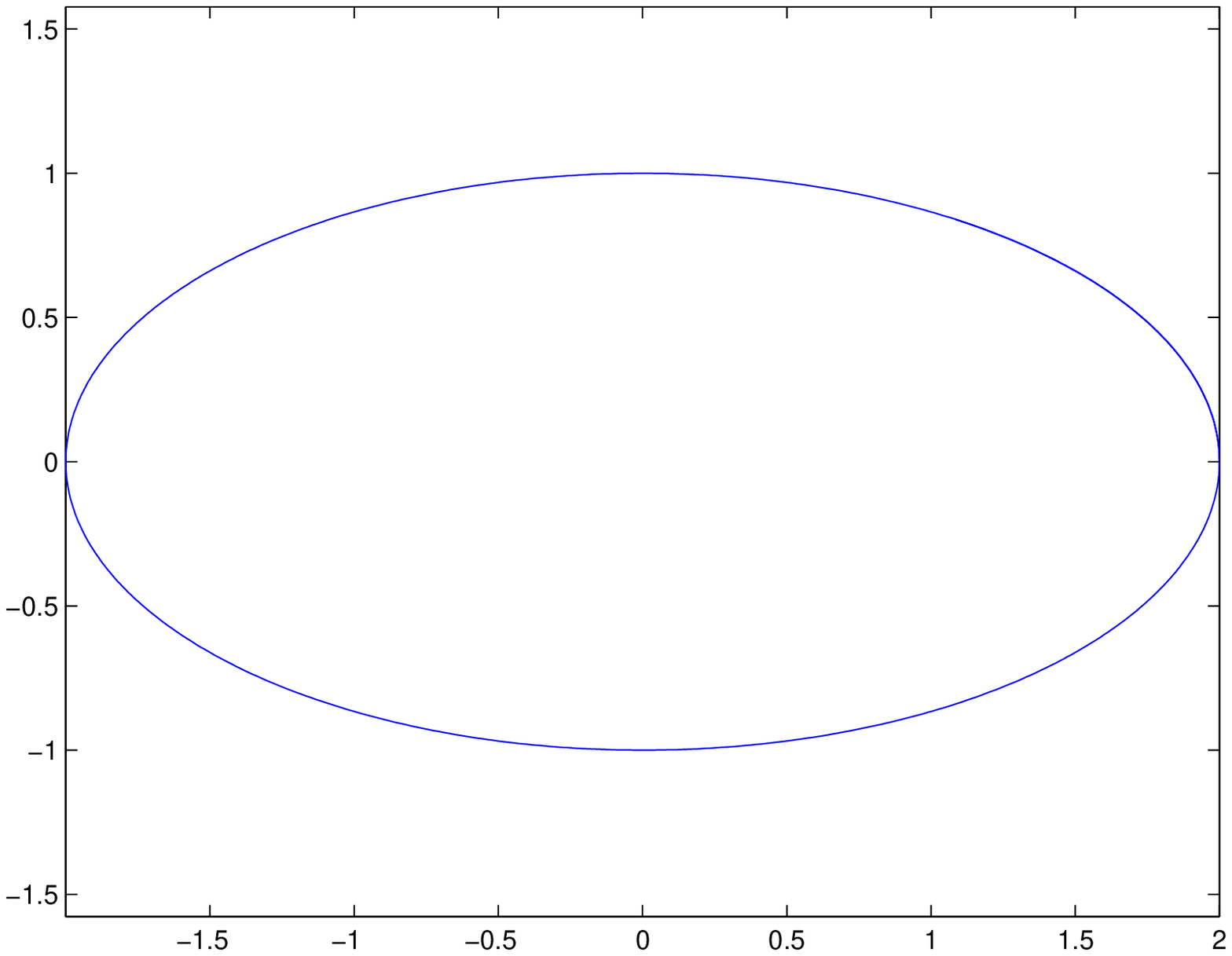}\hfill{}
 \hfill{}\includegraphics[clip,width=0.45\textwidth]{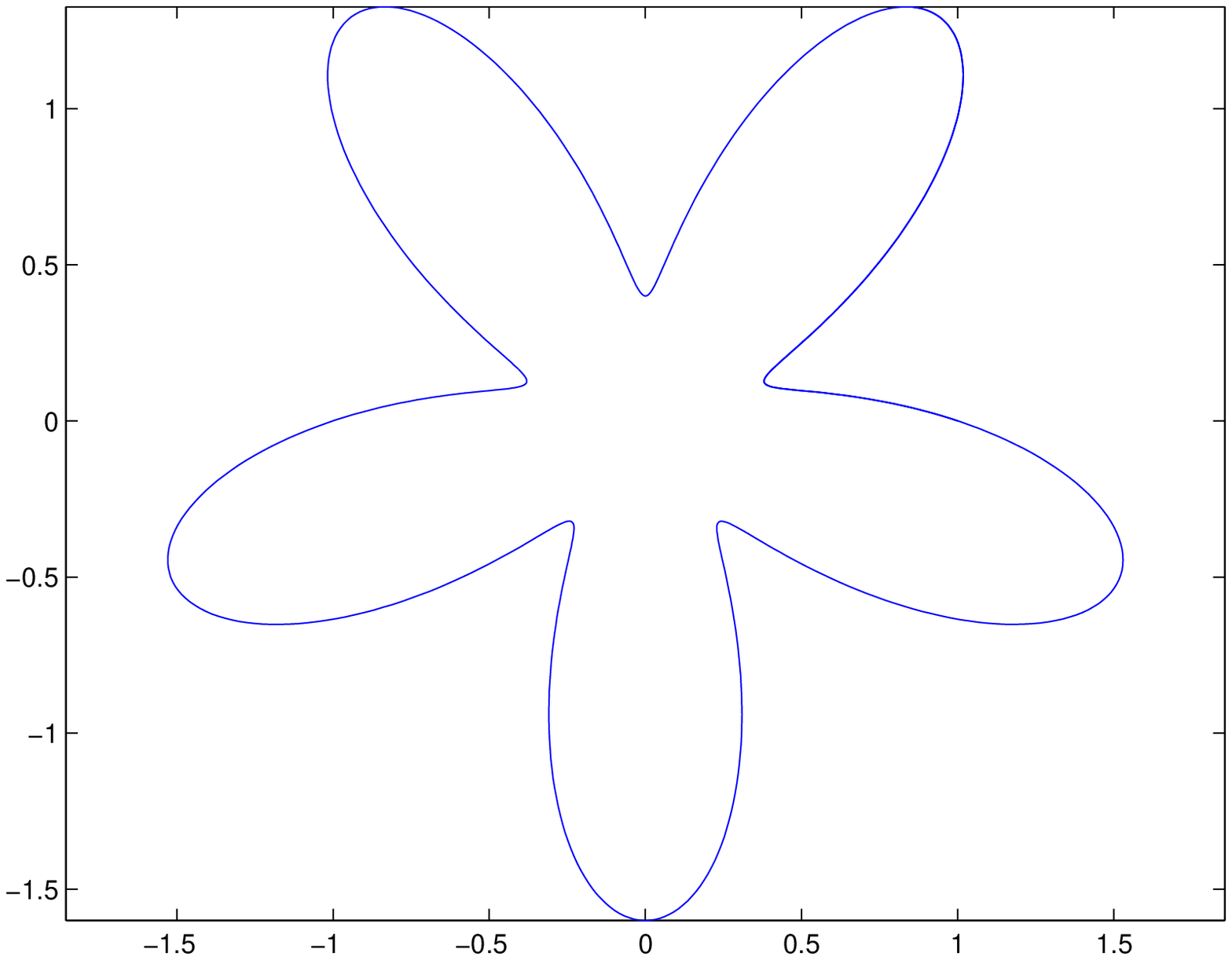}\hfill{}

 \hfill{}(a)\hfill{} \hfill{}(b)\hfill{}

\hfill{}\includegraphics[clip,width=0.45\textwidth]{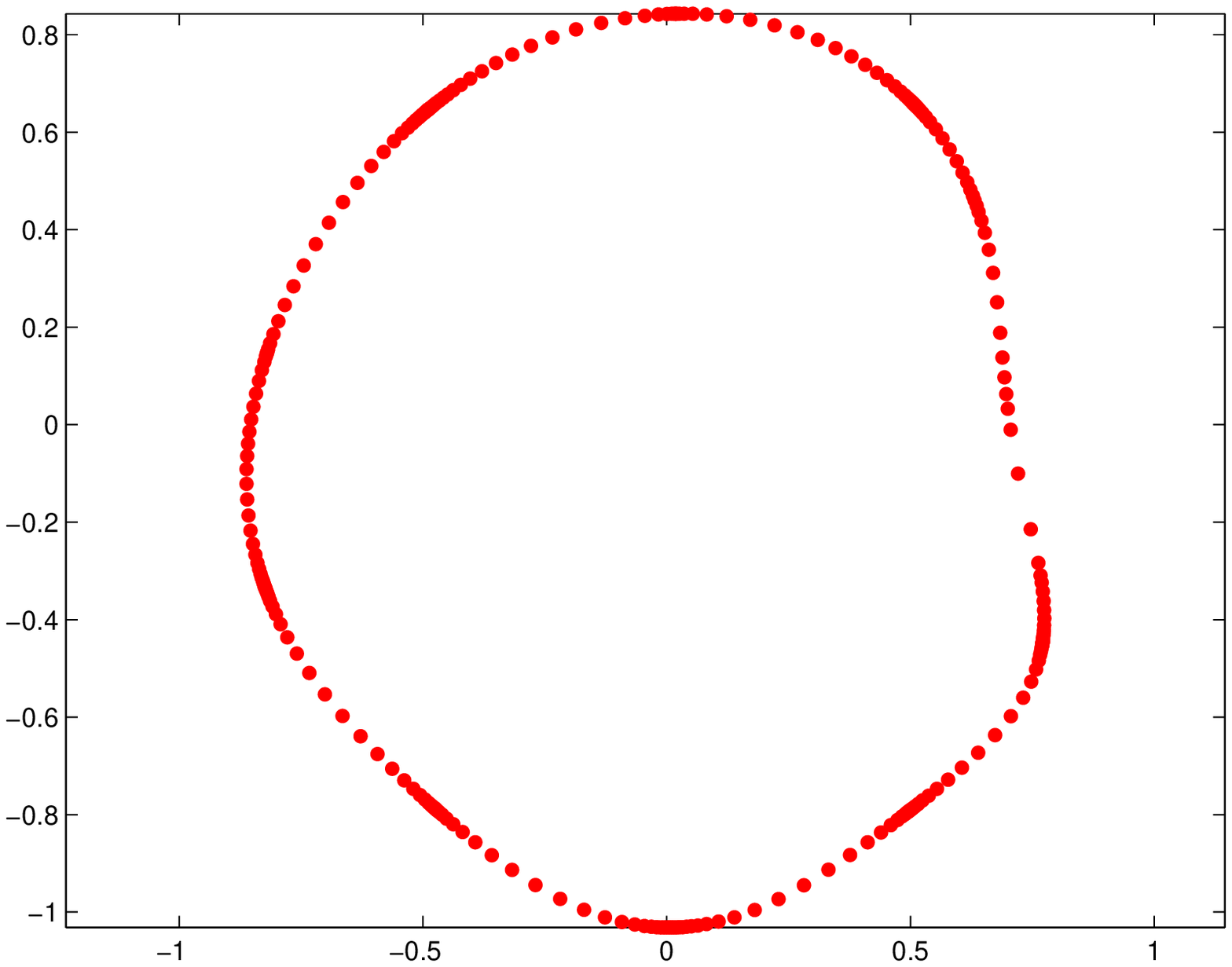}\hfill{}
 \hfill{}\includegraphics[clip,width=0.45\textwidth]{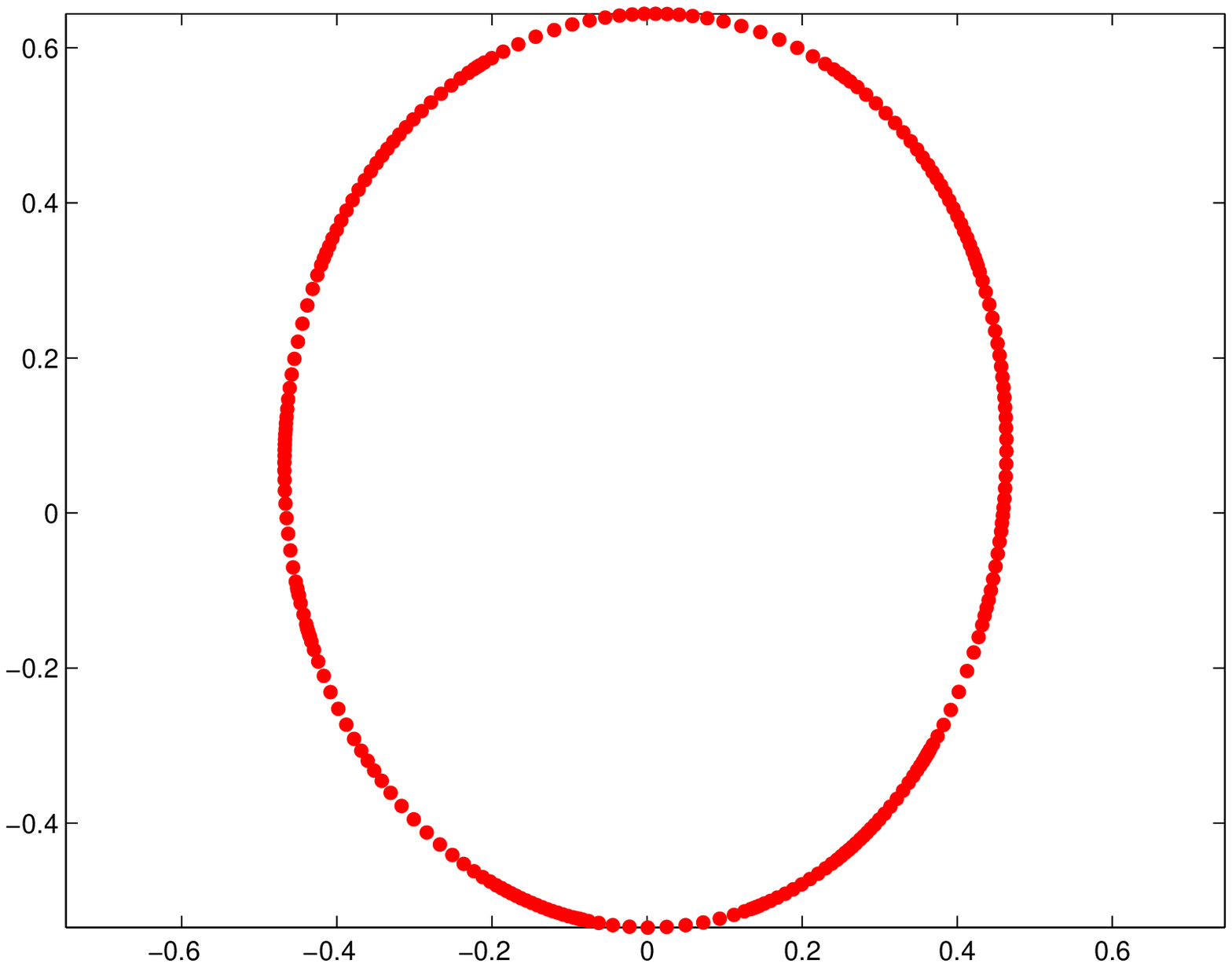}\hfill{}

 \hfill{}(c)\hfill{} \hfill{}(d)\hfill{}
 \vskip -0.15truecm
 \caption{\small (a): target shape in Example 1;  (b): initial guess;
 (c) and (d):  two reconstructed shapes that appear most frequently with 1\% random noise.}
 \label{invellipse}
 \end{figurehere}

\ss\no \textbf{Example 2}.
In this example, our target shape is a heart-shaped domain of the form (\ref{circle_perturb_form}) with $\delta = 0.8$, $m = 1$; see Figure \ref{invheart}(a).  Starting with a very poor initial guess, a shape of the form (\ref{circle_perturb_form}) with $\delta = 0.6$, $m = 7$ (see Figure \ref{invheart}(b)), two reconstructed shapes that appear most frequently
with different sets of random noise are shown in Figures \ref{invheart}(c) and \ref{invheart}(d).
Considering the invariance of the target shape up to translation, rotation, and scaling, our reconstructions
seem to be rather satisfactory.

 \begin{figurehere}
 \hfill{}\includegraphics[clip,width=0.45\textwidth]{./circle_perturb}\hfill{}
 \hfill{}\includegraphics[clip,width=0.45\textwidth]{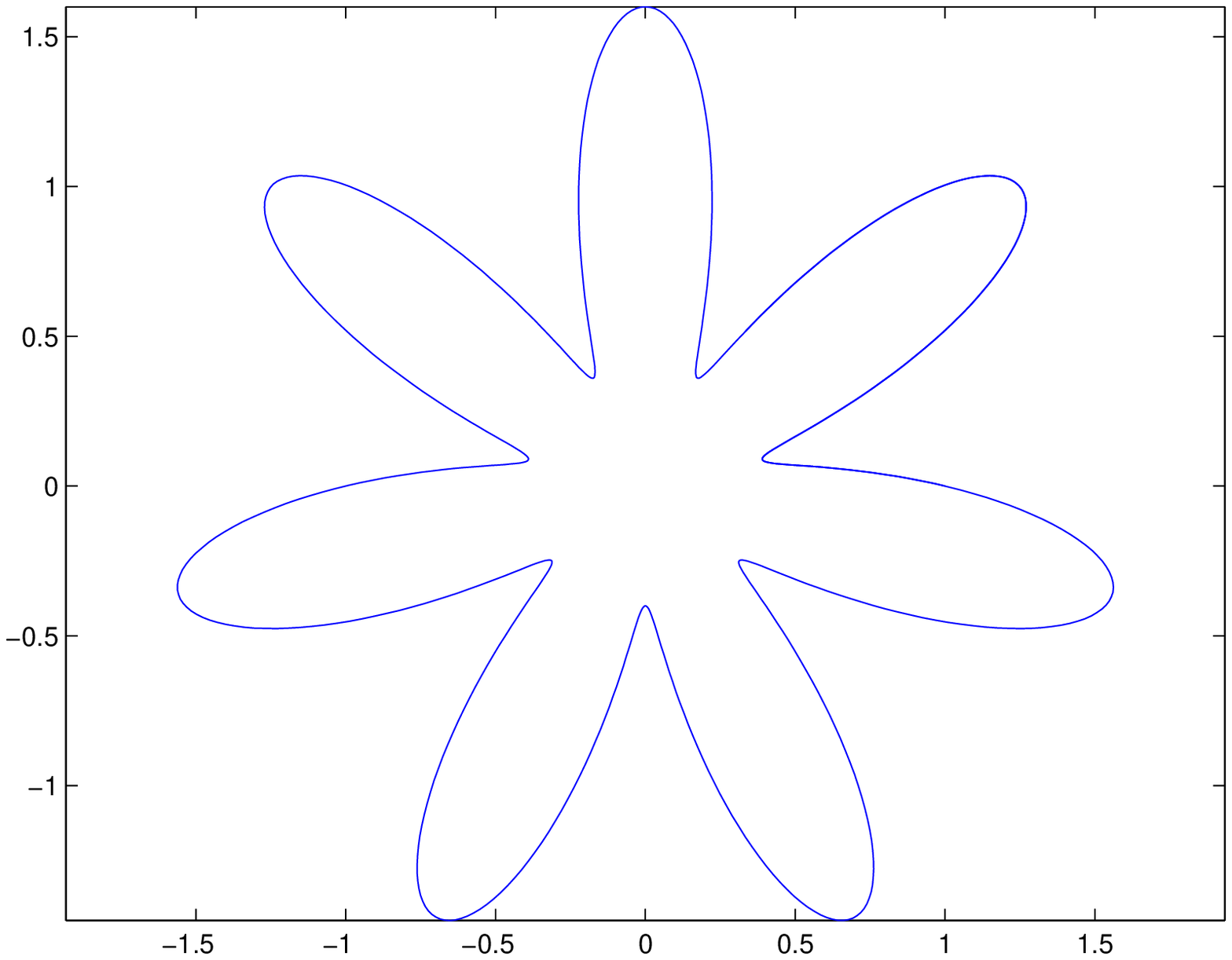}\hfill{}

 \hfill{}(a)\hfill{} \hfill{}(b)\hfill{}

\hfill{}\includegraphics[clip,width=0.45\textwidth]{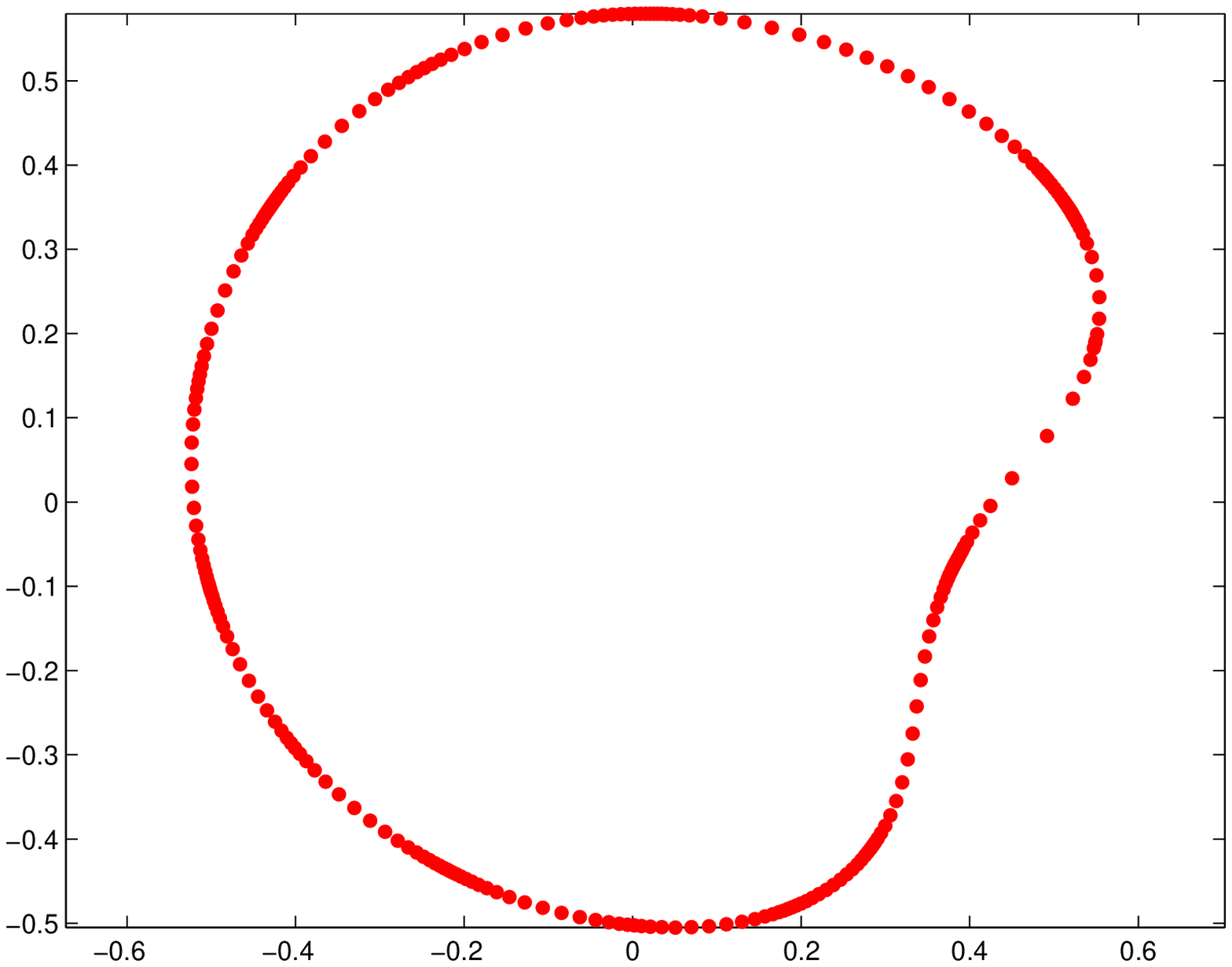}\hfill{}
 \hfill{}\includegraphics[clip,width=0.45\textwidth]{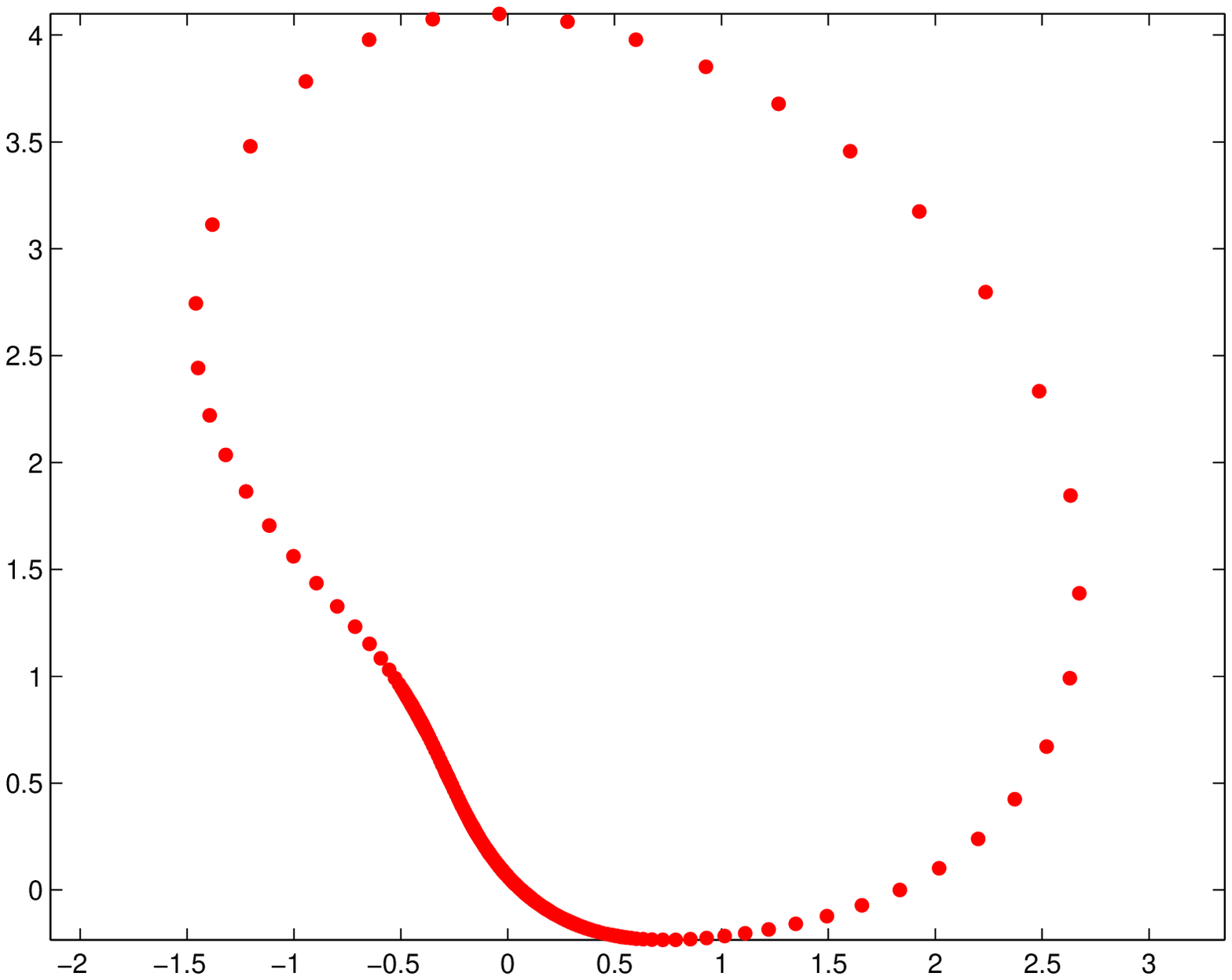}\hfill{}

 \hfill{}(c)\hfill{} \hfill{}(d)\hfill{}
  \vskip -0.15truecm
 \caption{\small (a): target shape in Example 2;  (b): initial guess;
 (c) and (d):  two reconstructed shapes that appear most frequently with 1\% random noise.}
 \label{invheart}
 \end{figurehere}

\ss\no\textbf{Example 3}.
A peanut-shaped domain of the form (\ref{circle_perturb_form}) with $\delta = 0.6$, $m = 2$ is investigated in this example; see Figure \ref{invpeanut}(a). Our initial guess is of the form (\ref{circle_perturb_form}) with $\delta = 0.6$, $m = 5$; see Figure \ref{invpeanut}(b).  Figure \ref{invpeanut} (c) and Figure \ref{invpeanut}(d) present two reconstructed shapes that appear most frequently
with different sets of random noise.

 \begin{figurehere}
 \hfill{}\includegraphics[clip,width=0.45\textwidth]{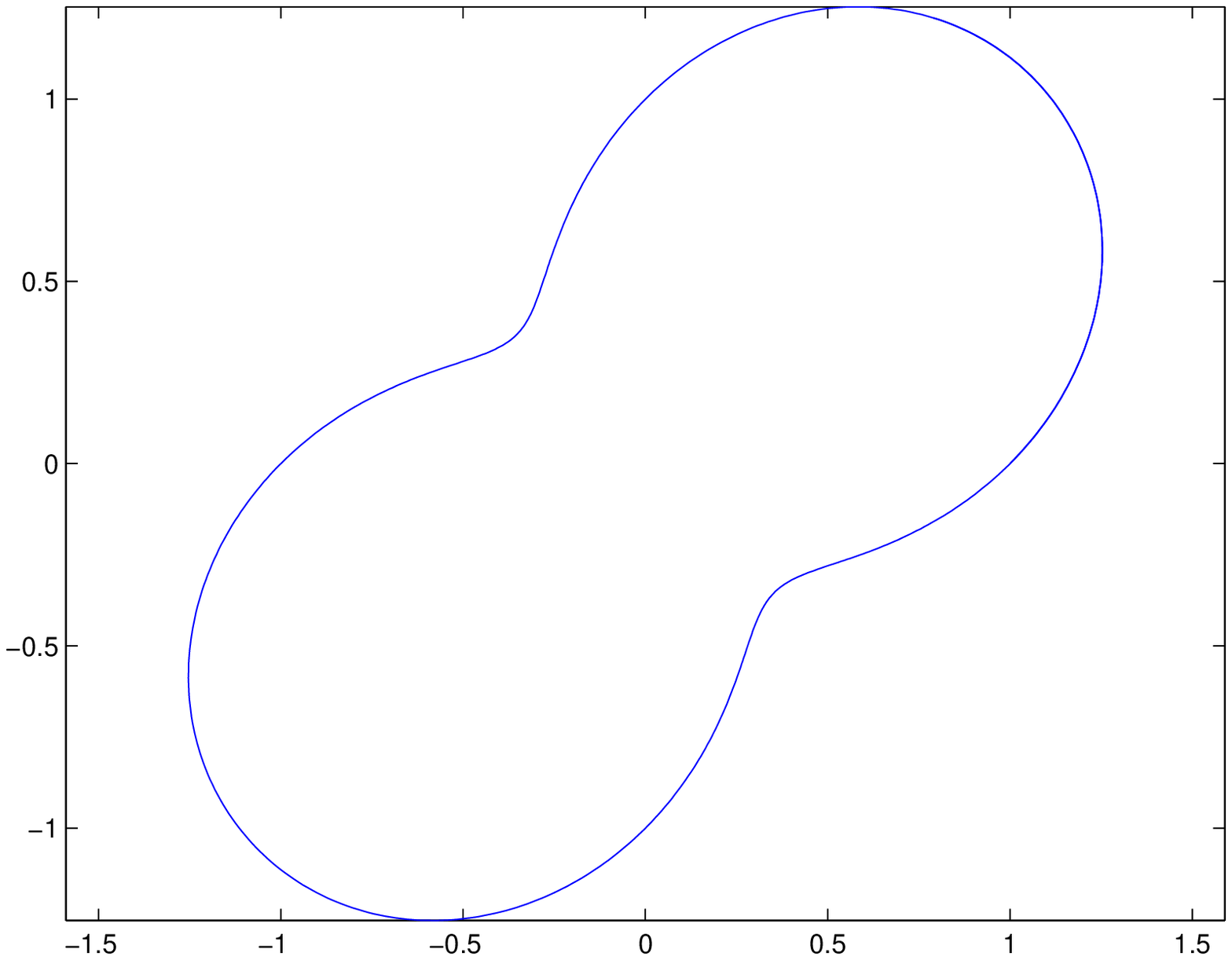}\hfill{}
 \hfill{}\includegraphics[clip,width=0.45\textwidth]{./floriform5petals}\hfill{}

 \hfill{}(a)\hfill{} \hfill{}(b)\hfill{}

\hfill{}\includegraphics[clip,width=0.45\textwidth]{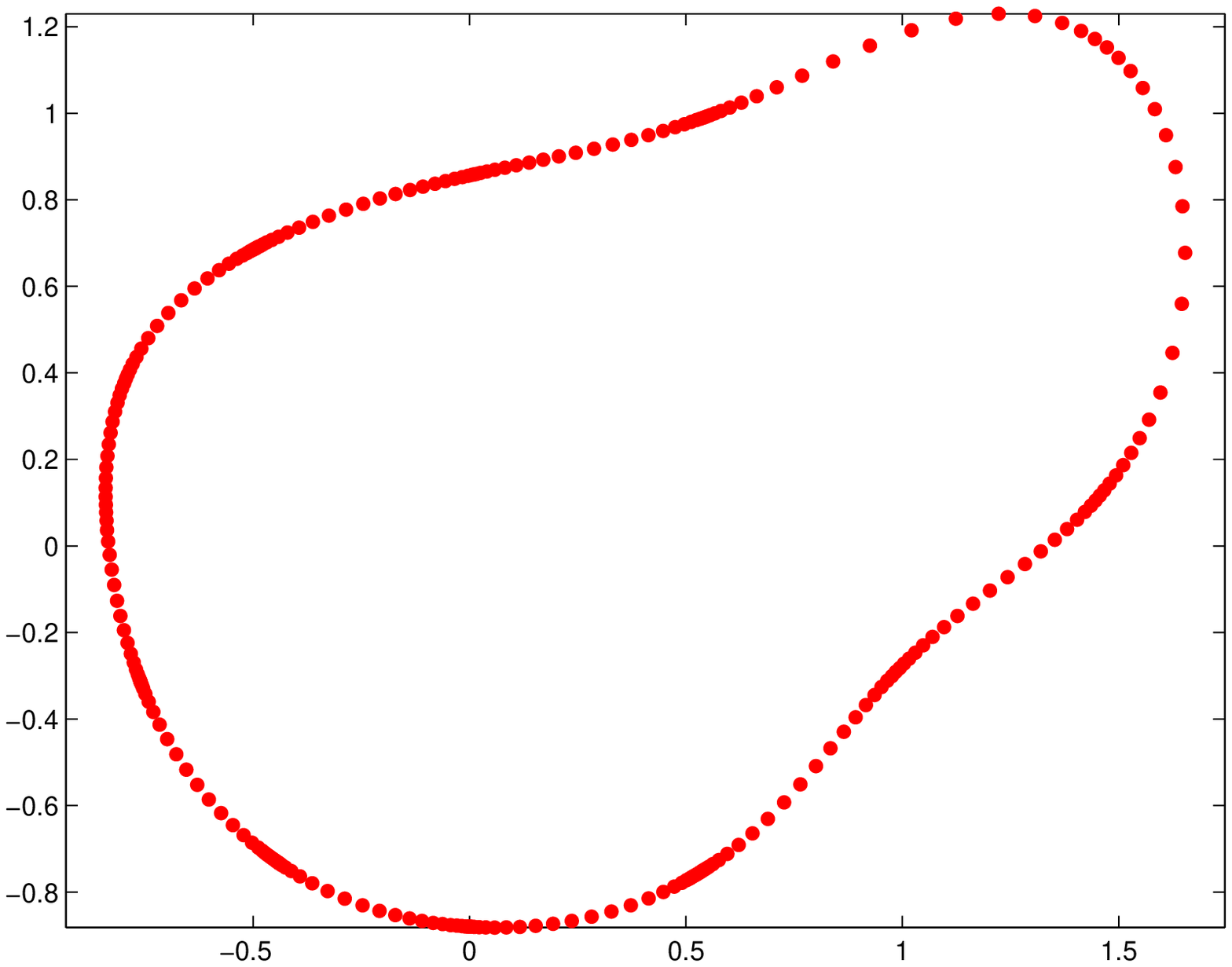}\hfill{}
 \hfill{}\includegraphics[clip,width=0.45\textwidth]{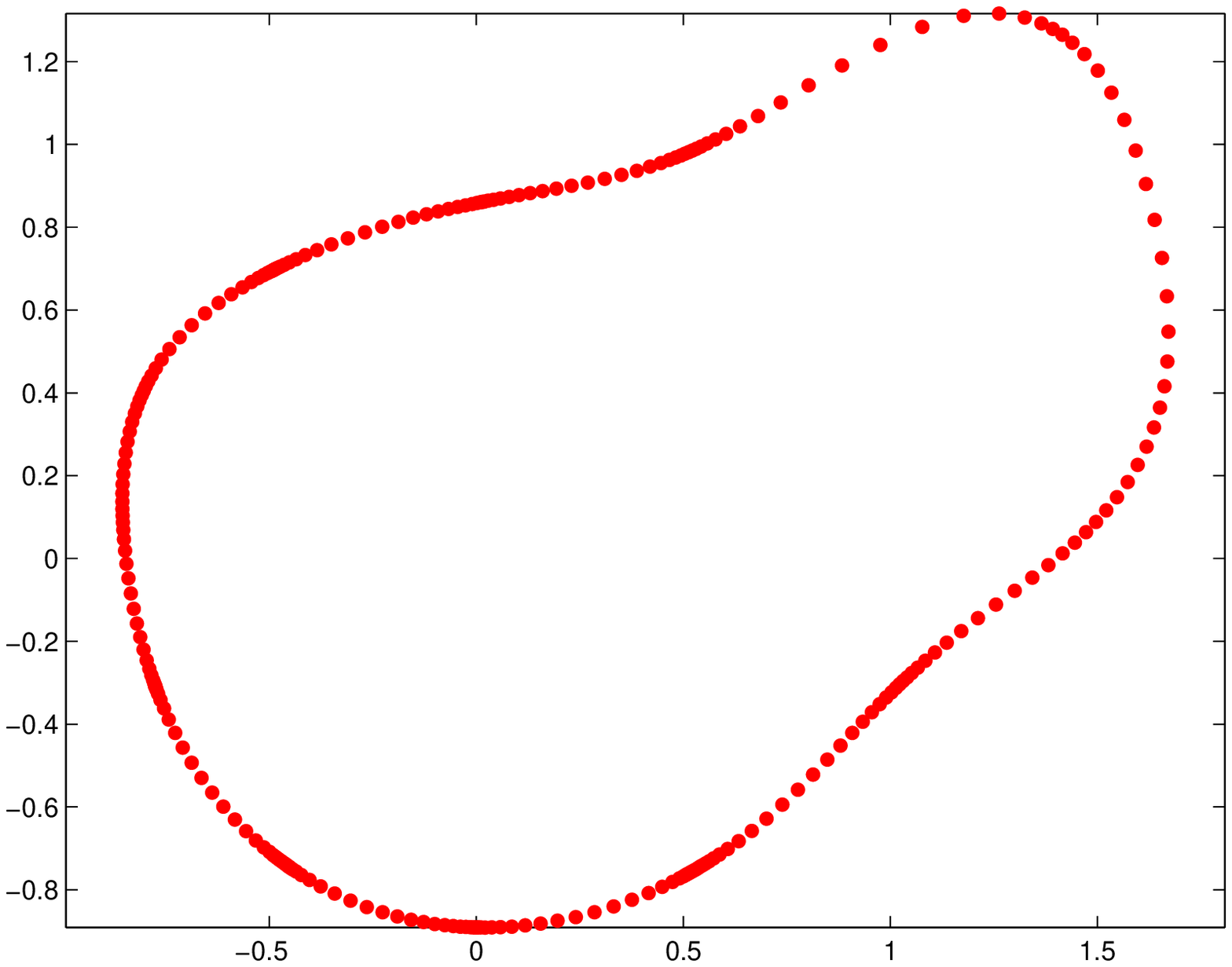}\hfill{}

 \hfill{}(c)\hfill{} \hfill{}(d)\hfill{}
 \vskip -0.15truecm
 \caption{\small (a): target shape in Example 3;  (b): initial guess;
 (c) and (d):  two reconstructed shapes that appear most frequently with 1\% random noise.}\label{invpeanut}
 \end{figurehere}

\ss\no \textbf{Example 4}.
In this example, we consider a pear-shaped domain of the form (\ref{circle_perturb_form}) with $\delta = 0.3$, $m = 3$; see Figure \ref{invpear}(a).  We start from the initial guess of the form (\ref{circle_perturb_form}) with $\delta = 0.6$, $m = 3$; see Figure \ref{invpear}(b). The reconstructed shapes that appear most frequently from the data polluted by different sets of random noise are shown in Figure \ref{invpear}(c) and  Figure \ref{invpear}(d).
Considering the random noise added in the spectral data and
the sensitivity of eigenvalue problem, our reconstructions prove to be
quite satisfactory.

\section{Concluding remarks}

In this work we have proposed numerical methods to recover the Fredholm eigenvalues of
a domain from the measurements of its polarization tensor at multiple contrasts or frequencies.
Then we have developed an optimal shape design algorithm (up to rigid transformations and scaling) based on partial knowledge of
Fredholm eigenvalues. Both inverse problems are highly nonlinear and severely ill-posed, but
our numerical experiments have demonstrated the effectiveness and robustness of the
proposed reconstruction algorithms. By using only the first few Fredholm eigenvalues, we have regularized the considered inverse problems. We expect that our results will have important applications in plasmon resonant nanoparticle design and in multifrequency imaging and classification of small anomalies from electrical capacitance measurements.

 \begin{figurehere}
 \hfill{}\includegraphics[clip,width=0.45\textwidth]{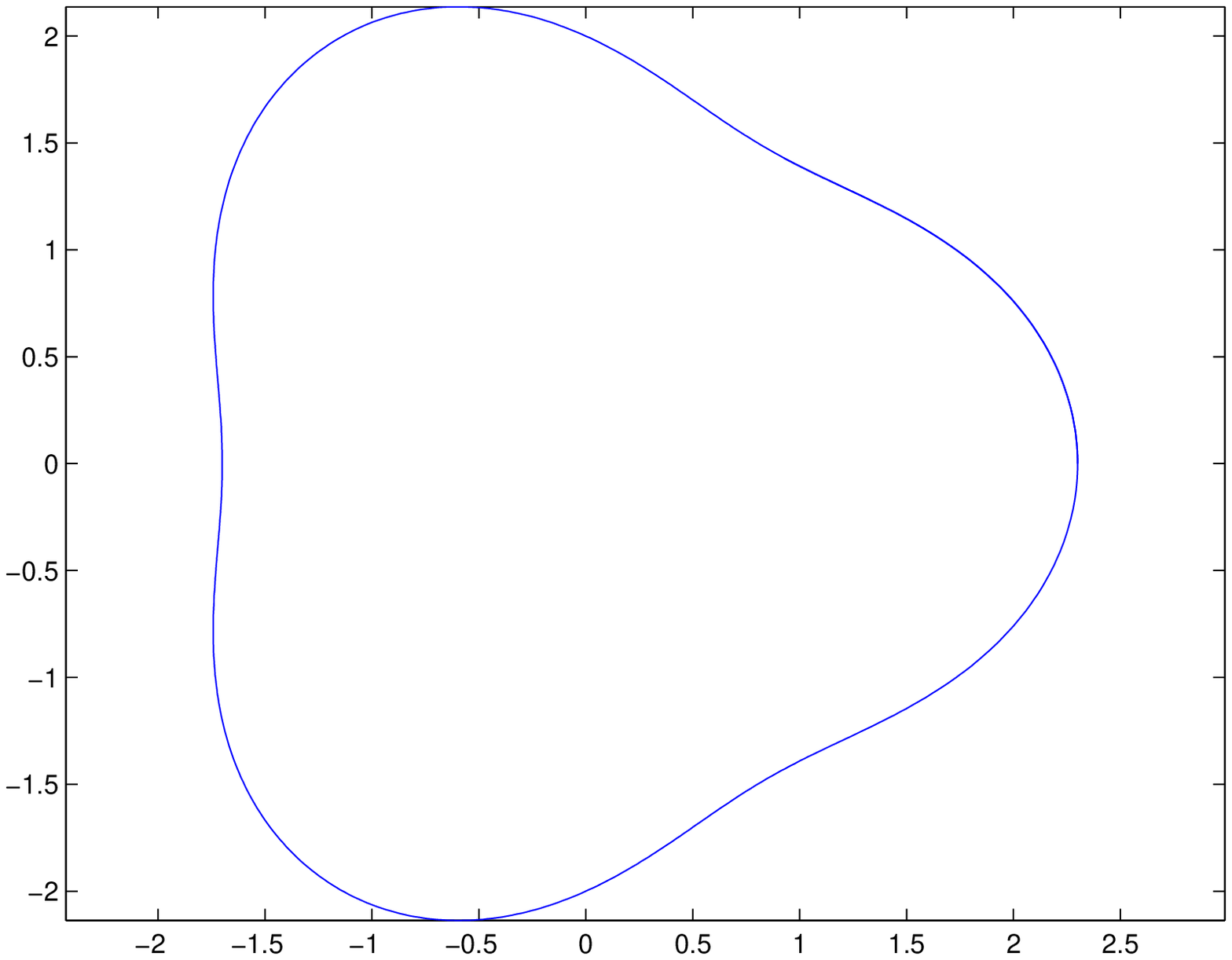}\hfill{}
 \hfill{}\includegraphics[clip,width=0.45\textwidth]{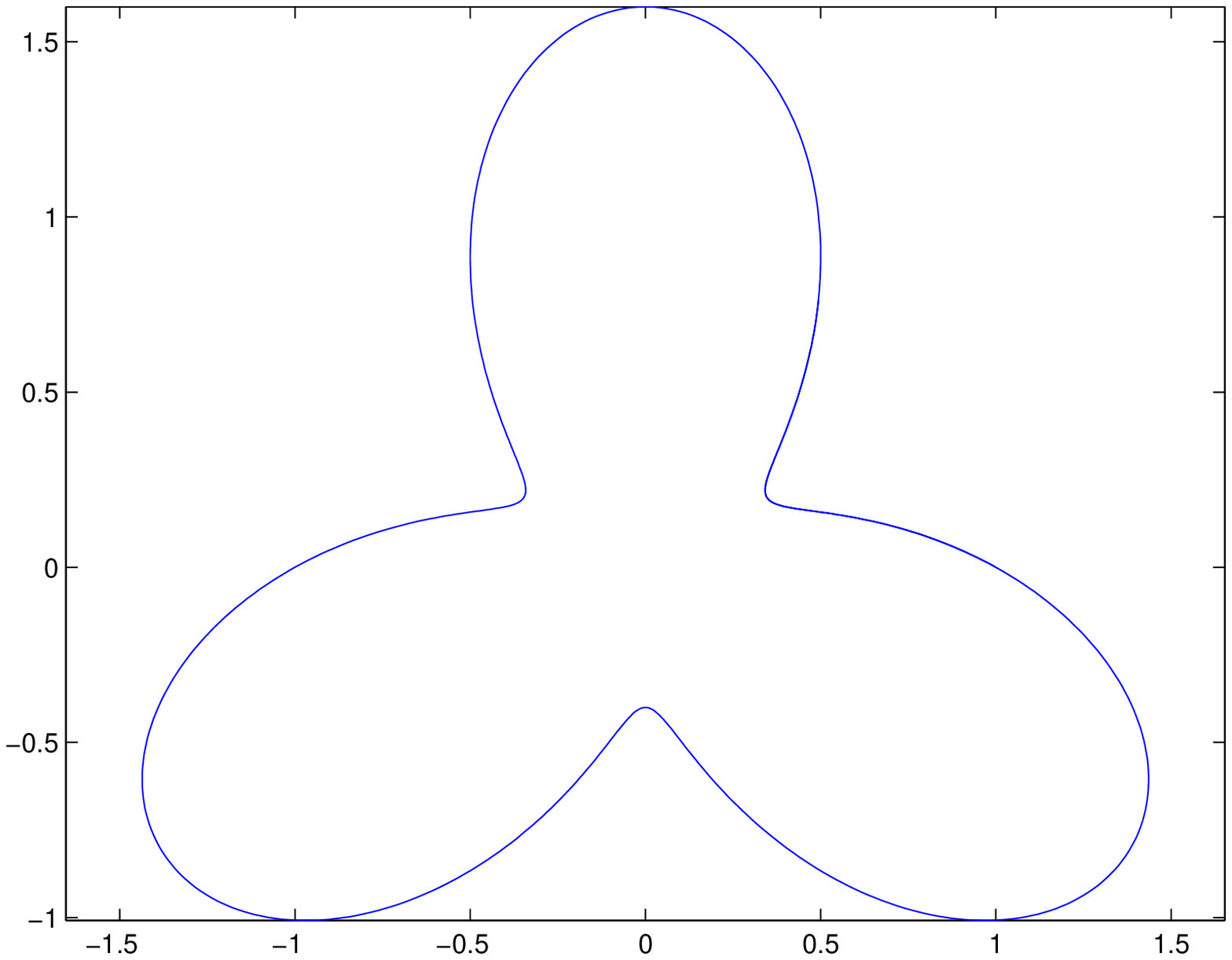}\hfill{}

 \hfill{}(a)\hfill{} \hfill{}(b)\hfill{}

\hfill{}\includegraphics[clip,width=0.45\textwidth]{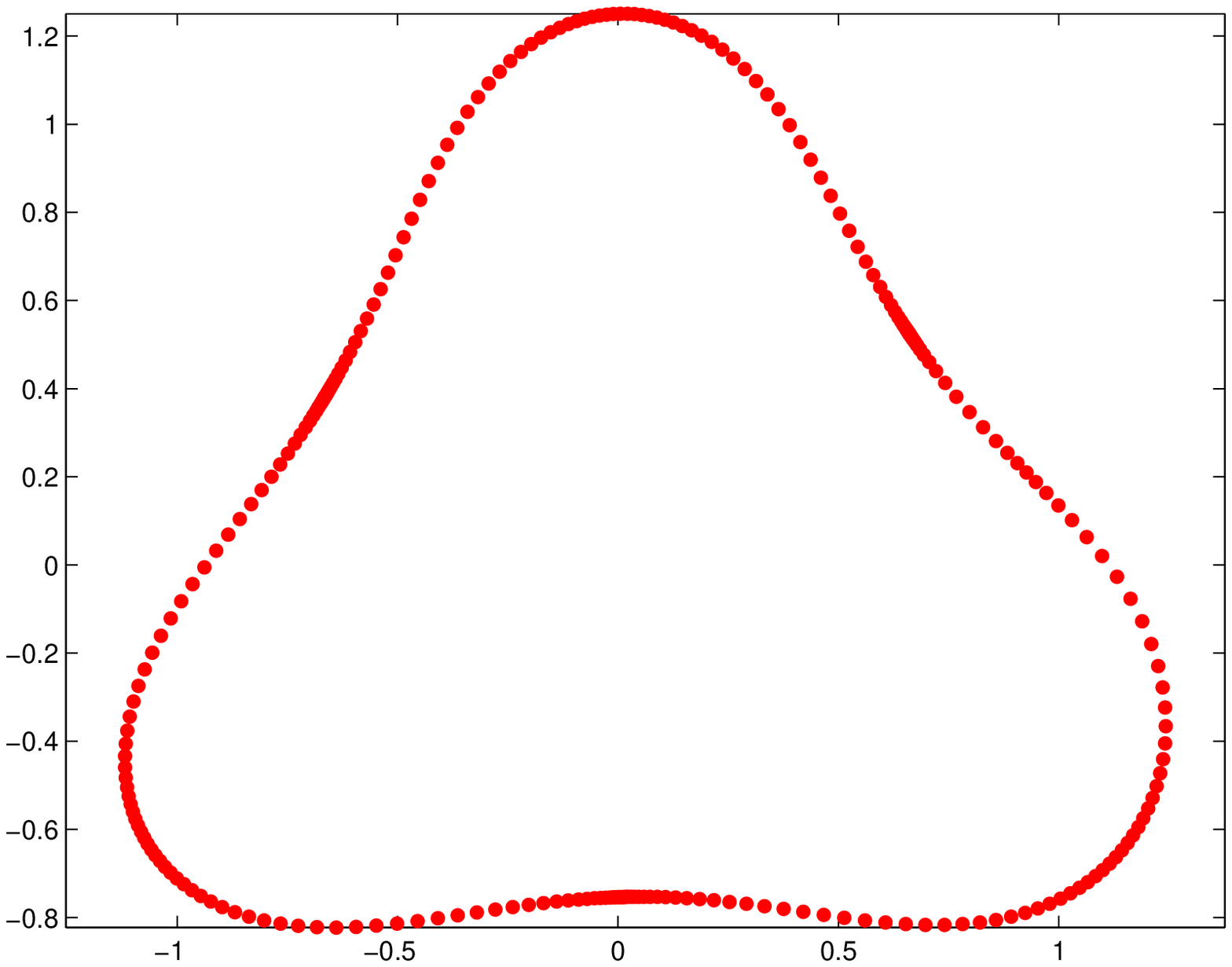}\hfill{}
 \hfill{}\includegraphics[clip,width=0.45\textwidth]{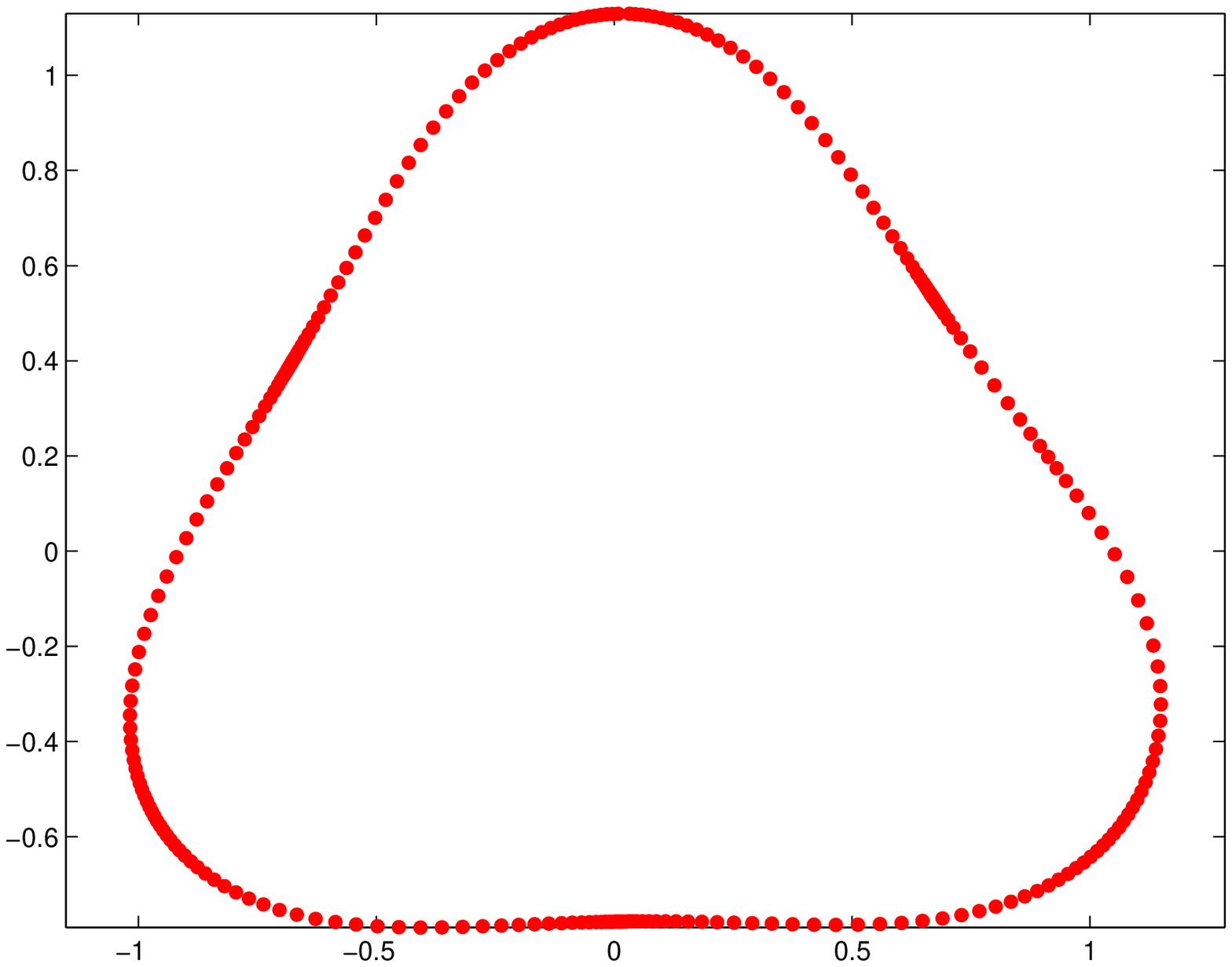}\hfill{}

 \hfill{}(c)\hfill{} \hfill{}(d)\hfill{}
 \vskip -0.15truecm
 \caption{\small (a): target shape in Example 4;  (b): initial guess;
 (c) and (d):  two reconstructed shapes that appear most frequently with 1\% random noise.}
 \label{invpear}
 \end{figurehere}

\appendix

\section{Pulse shape design} \label{curve}

In this section, we show how to acquire the PT at multiple contrasts in electrical capacitance
tomography using pulsed imaging.

Given a harmonic function $u_0$ in $\mathbb{R}^d$, a final time $T>0$, and pulse shape $h(t)$, that is supposed to be a compactly supported function in $(0, T)$, electrical capacitance
tomography is to reconstruct the inclusion $D$ from measurements of the solution $u(x,t)$ to the following system
\beqn
    \begin{cases}
        \nabla \cdot \bigg( \varepsilon_{D} + \varepsilon_c^\prime \partial_t \bigg) \nabla u = 0 &\text{ in }\; \mathbb{R}^d \times (0, T), \\[1.5mm]
         u(x,t) - u_0(x) h(t)= O(|x|^{1-d}) &\text{ as }\; |x| \rightarrow \infty, \mbox{for all } t \in (0,T),
    \end{cases}
    \label{transmissionh}
\eqn
where $\varepsilon_c^\prime$ is a positive constant. Here,  $\varepsilon_c$ and $\varepsilon_c^\prime$ are respectively the conductivity and permittivity of $D$. The background medium $\mathbb{R}^d\setminus \overline D$ is assumed to be with conductivity $\varepsilon_m$ and $0$ permittivity.
In the time-harmonic regime, we call $\varepsilon_c + i \omega \varepsilon_c^\prime$ the admittivity of  $D$ with $\omega$ being the operating frequency.

Let $\sigma := \varepsilon_c/\varepsilon_m$ and $\varepsilon := \varepsilon_c^\prime/\varepsilon_m$.
According to \cite{book}, we can reconstruct the polarization tensor
$M(\lambda(t), D)$ from the measurements of $u$ for $x$ being far away from $D$,
where
\beqn
    \lambda(t) = \f{(\sigma +1) h(t) + \varepsilon h'(t)}{2(\sigma -1) h(t) + 2\varepsilon h'(t)} \, , \quad t\in (0,T)\,.
    \label{curvelambda}
\eqn
From the above formula, we can see that a different pulse $h$ gives a different curve $\gamma := \{ \lambda(t) \in \mathbb{C} : t\in (0,T)\}$ on the complex plane $\mathbb{C}$.  Motivated from Section\,\ref{GPT}, we aim to find a pulse $h \in \mathcal{C}^\infty_c(0,T)$ for some $T>0$ such that the curve $\gamma := \{ \lambda(t) \in \mathbb{C} : t\in (0,T)\}$ given by \eqref{curvelambda} encloses the spectrum of $\sigma(\mathcal{K}_{\partial D}^*) \subset (-1/2,1/2]$.

Therefore, we shall investigate different possible shapes of the impulse $h$ which gives an optimal shape of the curve $\gamma$. Our desired $\gamma$ should be a smooth simple closed curve enclosing $(-0.5,0.5]$.  Then we can recover the eigenvalues of $\sigma(\mathcal{K}_{\partial D}^*)$ from $M(\lambda,\partial D)$ as in subsection \ref{func}.

By explicit calculations, if we let
\beqn
    p = \f{(\sigma +1) - 2 (\sigma -1) \lambda}{\varepsilon (1- 2 \lambda)} \, ,
\eqn
then we have from \eqref{curvelambda} the following explicit form for the pulse $h$:
\beqn
h(t) = C \exp\left(\int_0^t p(s)ds\right),
\eqn
where $C$ is a constant. Letting $\lambda = A e^{2 \pi i t}$, we get
\beqn
h(t) = e^{-2t} (2 A e^{2 \pi i t} - 1)^{-\f{i}{2 \pi}}\,,
\label{pulseh}
\eqn
where the function $z^i$ is defined as $z^i := e^{i \log x}$. Note that although the solution in \eqref{pulseh} is not necessarily
compactly supported, we can always extend the function to a compactly supported smooth function on $(-\varepsilon, T + \varepsilon)$ for some $\varepsilon >0$.

Given parameters $\sigma = 3$, $\varepsilon = 2$ and $T = 20$, we have tried different shapes of the pulse $h(s)$.


\noindent \textbf{Example 1}\; In this example we choose $h(t) = A(t) \phi(t)$, where $\phi(t) = \exp(\f{(t-a)\pi}{2 \sigma_0}i+\f{\pi}{2}i)$ and $A(t) = \f{1}{\sqrt{2 \pi}\sigma_0}\exp\left(-\f{(t-a)^2}{2 \sigma_0^2}\right)$ with $\sigma_0 = 0.3, a = 3$. The real part of the curve $h(t)$ and its corresponding curve $\lambda(t)$ on the complex plane $\mathbb{C}$ according to (\ref{curvelambda}) are
shown in Figure \ref{test2}.

\begin{figurehere}
     \begin{center}
     \vskip -0.2truecm
        \scalebox{0.7}{\includegraphics{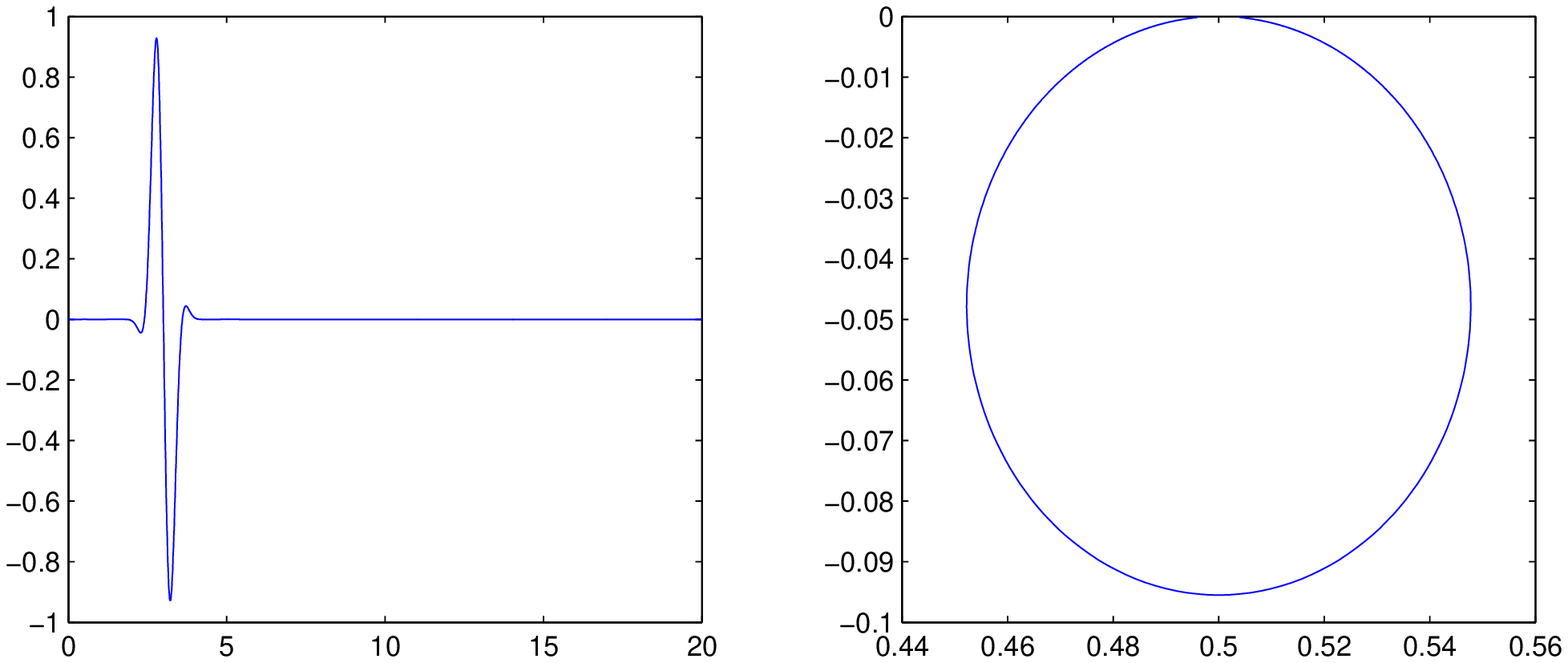}}\\
     \vskip -0.2truecm
     \caption{\small Real part of impulse $h(t)$ in Example 1 (left);
     Curve $\lambda(t)$ according to (\ref{curvelambda}) (right).}\label{test2}
     \end{center}
 \end{figurehere}

\noindent \textbf{Example 2}\; Choose
$h = -A(t) \phi(t)$, where $\phi(t) = \exp(i \cos(t-a) \pi)$ and $A(t) = \f{1}{\sqrt{2 \pi}\sigma_0}\exp\left(-\f{(t-a)^2}{2 \sigma_0^2}\right)$ with $\sigma_0 = 0.3, a = 3$. Figure \ref{test3} shows the real part of the curve $h(t)$ and its corresponding curve $\lambda(t)$ according to (\ref{curvelambda}).

\begin{figurehere}
     \begin{center}
     \vskip -0.2truecm
           \scalebox{0.7}{\includegraphics{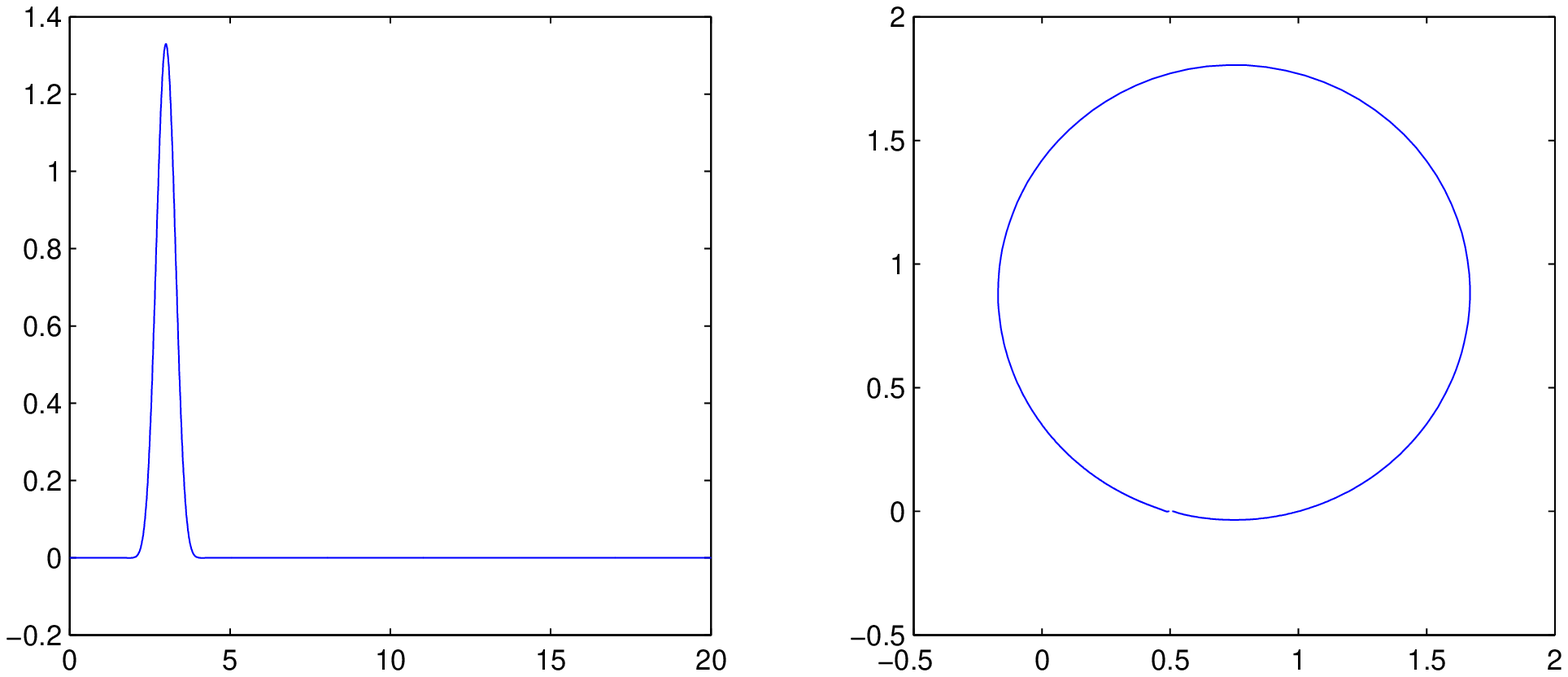}}\\
            \vskip -0.2truecm
     \caption{\small Real part of impulse $h(t)$ in Example 2 (left);
     Curve $\lambda(t)$ according to (\ref{curvelambda}) (right).}\label{test3}
     \end{center}
 \end{figurehere}

%

\noindent \textbf{Example 3}\; We choose $h = e^{-2t} (2 A e^{2 \pi i t} - 1)^{-\f{i}{2 \pi}}$. The real part of the curve $h(t)$ and its corresponding curve $\lambda(t)$ according to (\ref{curvelambda}) is shown on Figure \ref{test5}.

\begin{figurehere}
     \begin{center}
     \vskip -0.2truecm
           \scalebox{0.7}{\includegraphics{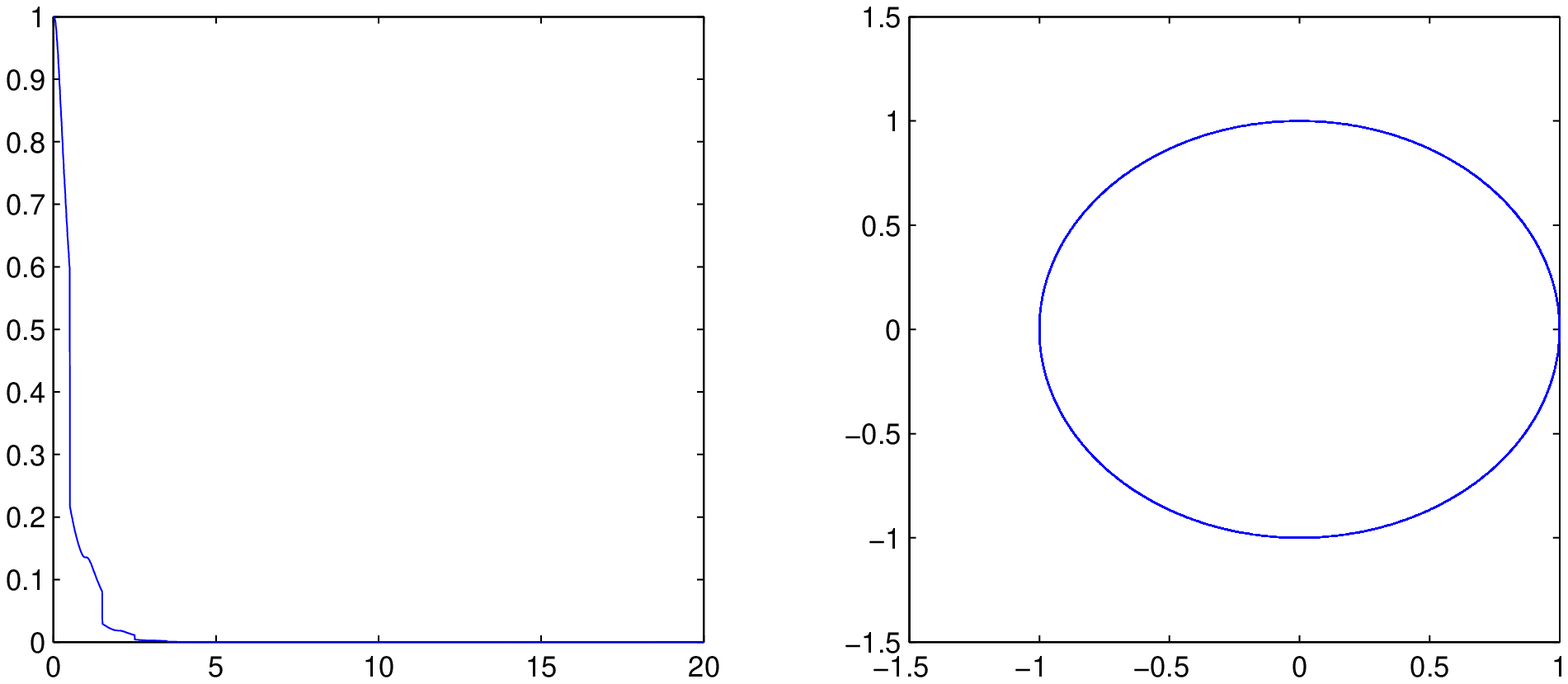}}\\
            \vskip -0.2truecm
     \caption{\small Real part of impulse $h(t)$ in Example 3 (left);
     Curve $\lambda(t)$ according to (\ref{curvelambda}) (right).}\label{test5}
     \end{center}
 \end{figurehere}

\section{Multiply connected objects} \label{multipleobject}
In this section, we briefly investigate the eigenvalue of the Neumann-Poincar\'e operator of a domain consisting of two identical copies of a non-overlapping shape with the same contrast.  Let the shape $D_1$ be given, we consider the shape
\beqn
D_v := D_1 \bigcup D_2,
\label{da}
\eqn
where $D_2 := D_1 + v$ and $v \in \mathbb{R}^2$ is such that the distance $d(D_1,D_2)$ between $D_1$ and $D_2$ is positive.
The Neumann-Poincar\'e operator $\mathbb{K}^*_{\partial D_v}$ associated with $D_v$ is given by \cite{ciraolo}
\beqn
     \mathbb{K}^*_{\partial D_v}:=
    \begin{pmatrix}
    \mathcal{K}^*_{\partial D_1} & \f{\partial}{\partial \nu_1} \mathcal{S}_{\partial D_2}\\
    \f{\partial}{\partial \nu_2} \mathcal{S}_{\partial D_1} &\mathcal{K}^*_{\partial D_2}
    \end{pmatrix} .
    \label{matrixda}
\eqn
We are interested in how the eigenvalues of $\mathbb{K}^*_{\partial D_v}$ behave as $v$ varies, and particularly when $d(D_1,D_2)\rightarrow 0$. As an example, we consider an ellipse, $D_1$, of the form \eqref{ellipsehaha}.
\begin{figurehere}
     \begin{center}
     \vskip -0.1truecm
           \scalebox{0.4}{\includegraphics{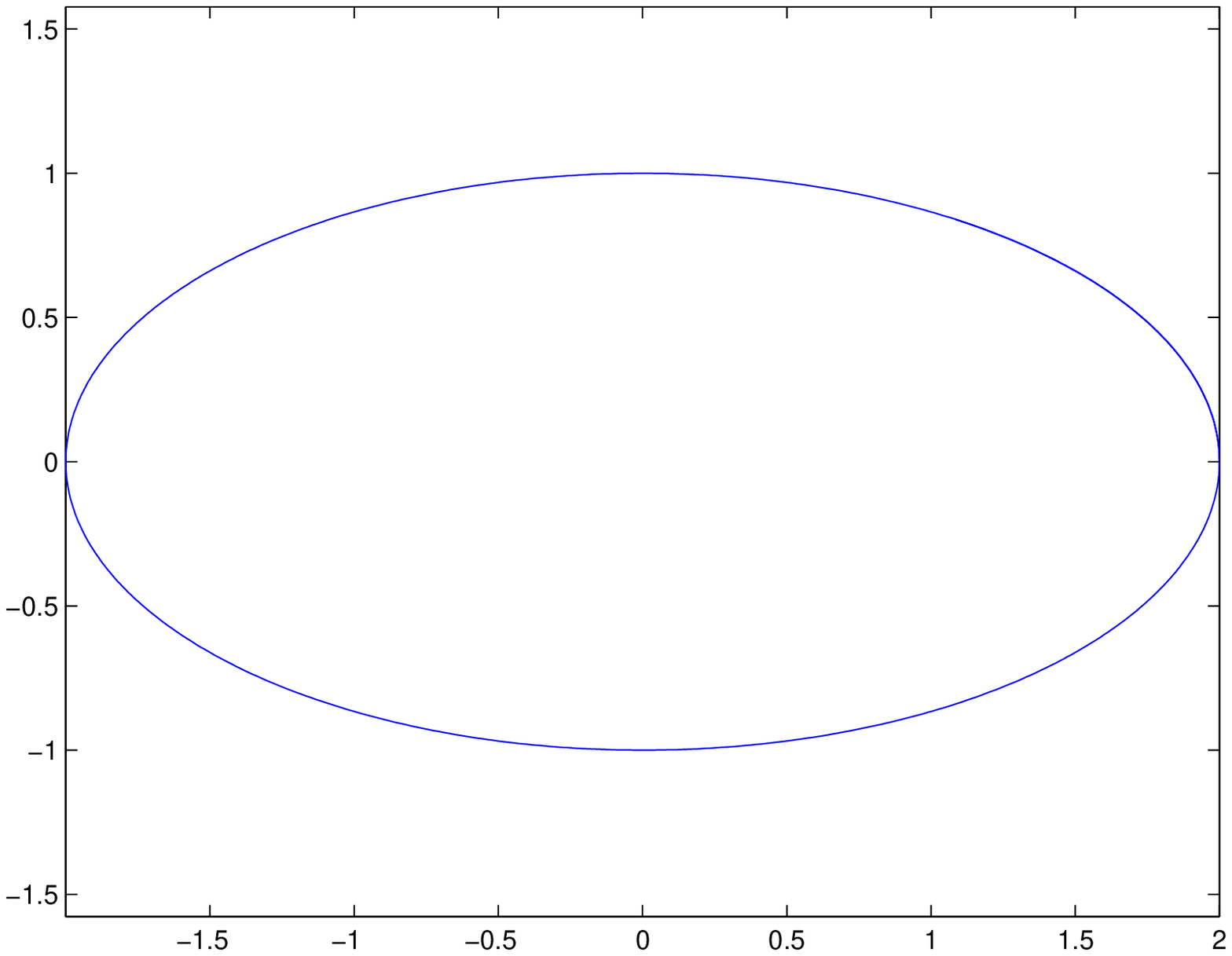}}\\
            \vskip -0.5truecm
     \caption{\small The ellipse $D_1$\,.}\label{ellipse}
     \end{center}
 \end{figurehere}
Letting $v = (2^k+2) (0,1) $
where $k=5-n$ and $n=1,2,\cdots,10$, we observe the change of the spectrum of $\mathbb{K}^*_{\partial D_v}$.  Figure \ref{multipletest} shows the eigenvalues of $\mathbb{K}^*_{\partial D_v}$ which is larger than $0.0005$ as $v$ varies.
\begin{figurehere}
        \hfill{}   \hfill{} \hfill{}   \hfill{}

        \hfill{}\includegraphics[clip,width=0.45\textwidth]{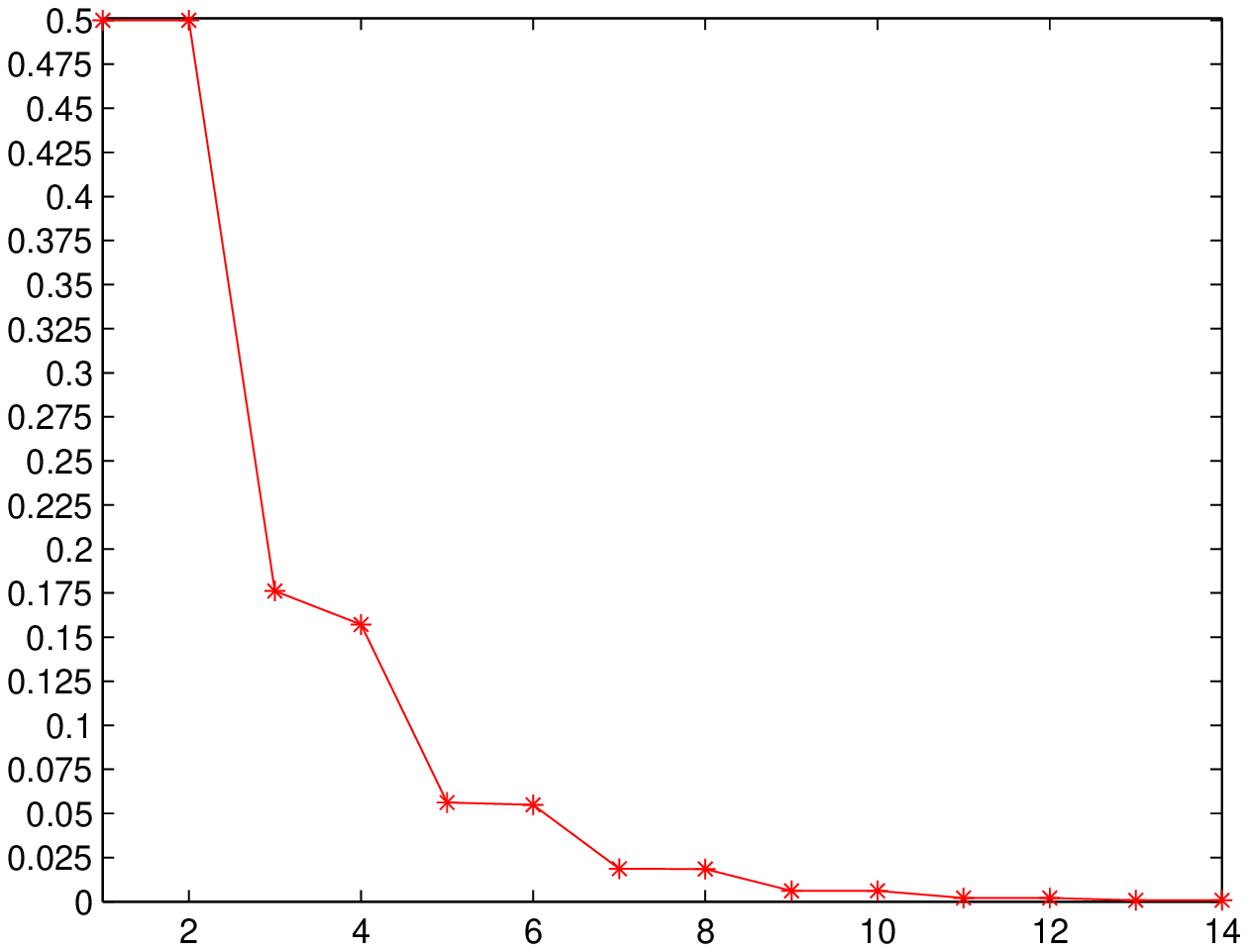}\hfill{}
        \hfill{}\includegraphics[clip,width=0.45\textwidth]{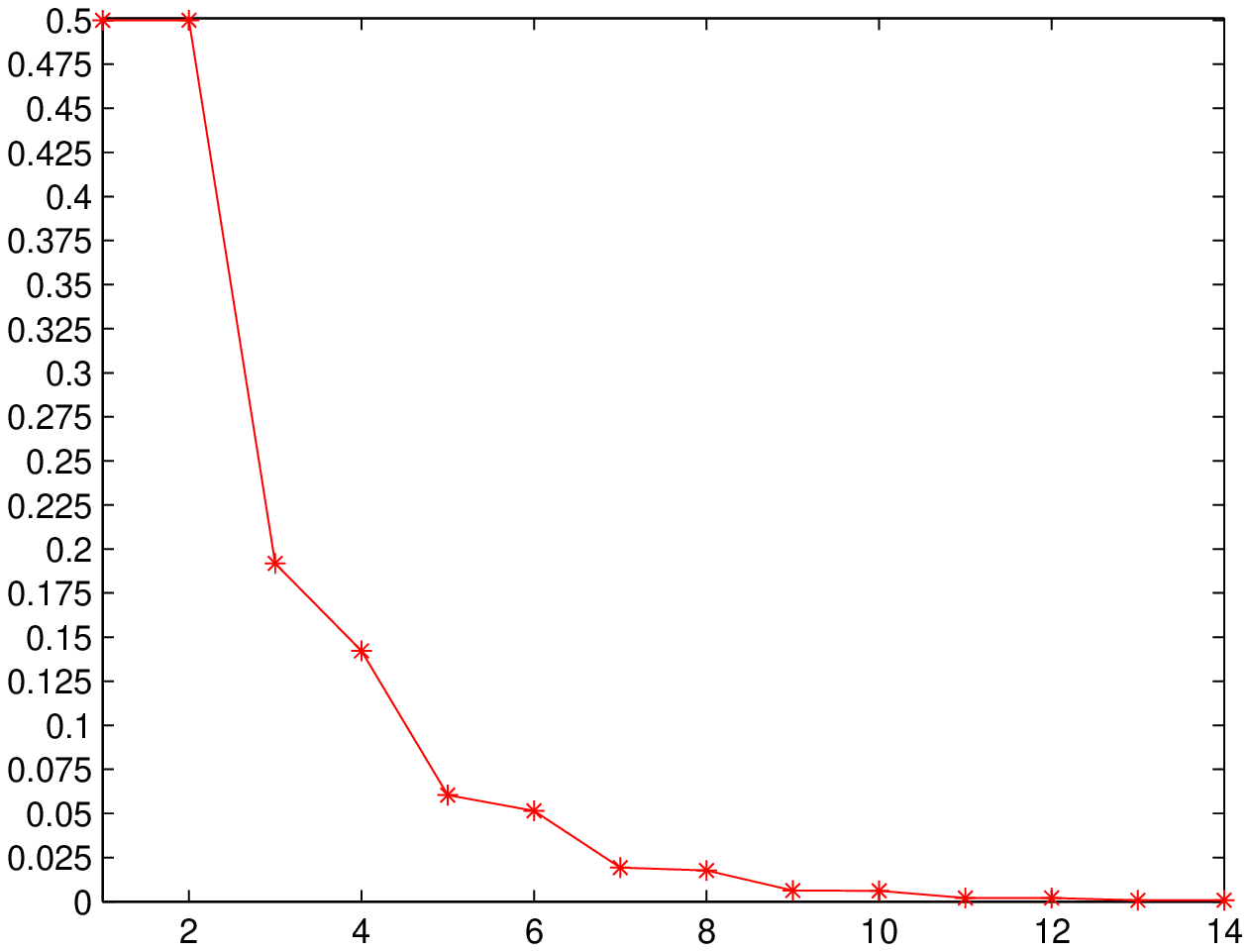}\hfill{}

        \hfill{}(a)\hfill{} \hfill{}(b)\hfill{}

        \hfill{}\includegraphics[clip,width=0.45\textwidth]{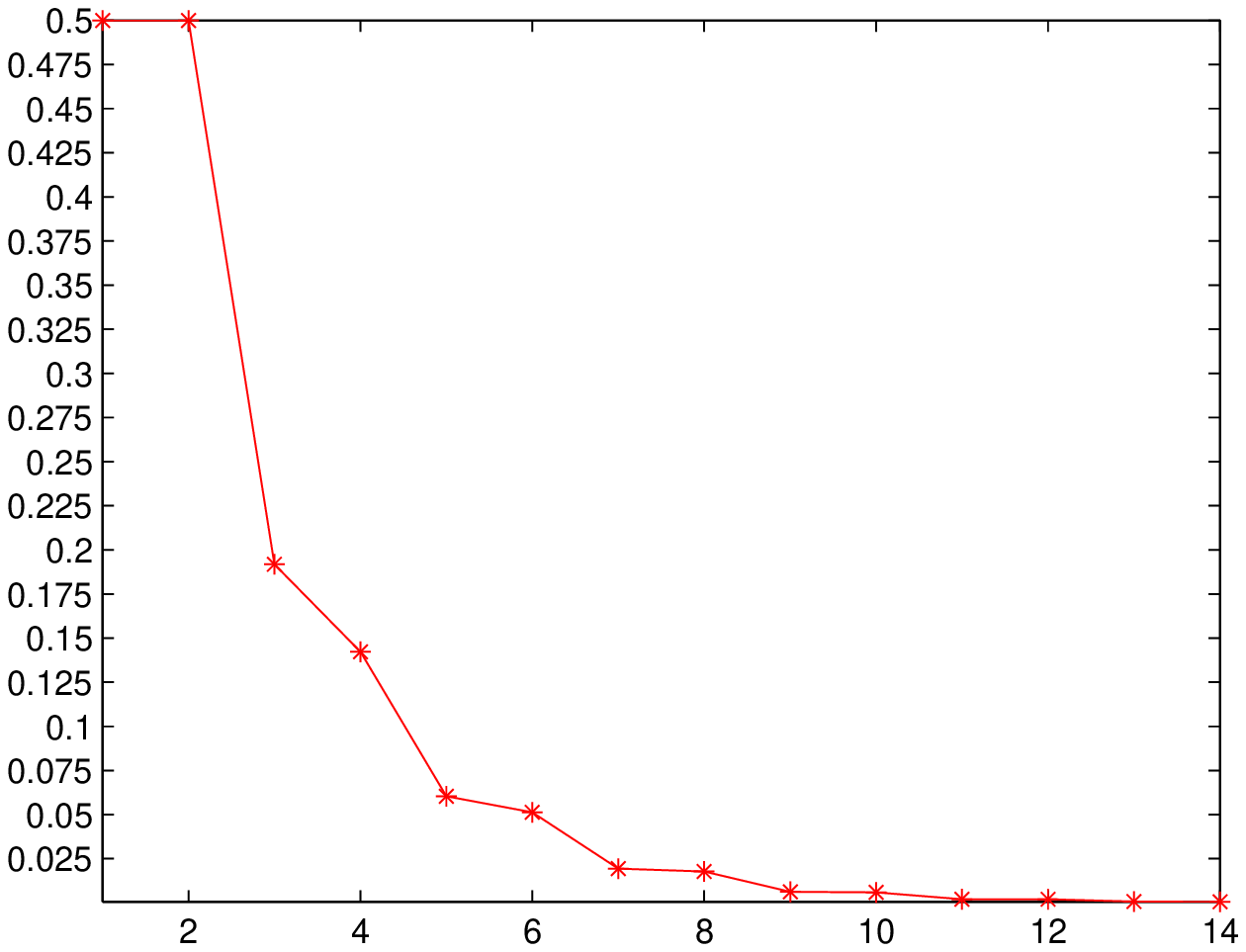}\hfill{}
        \hfill{}\includegraphics[clip,width=0.45\textwidth]{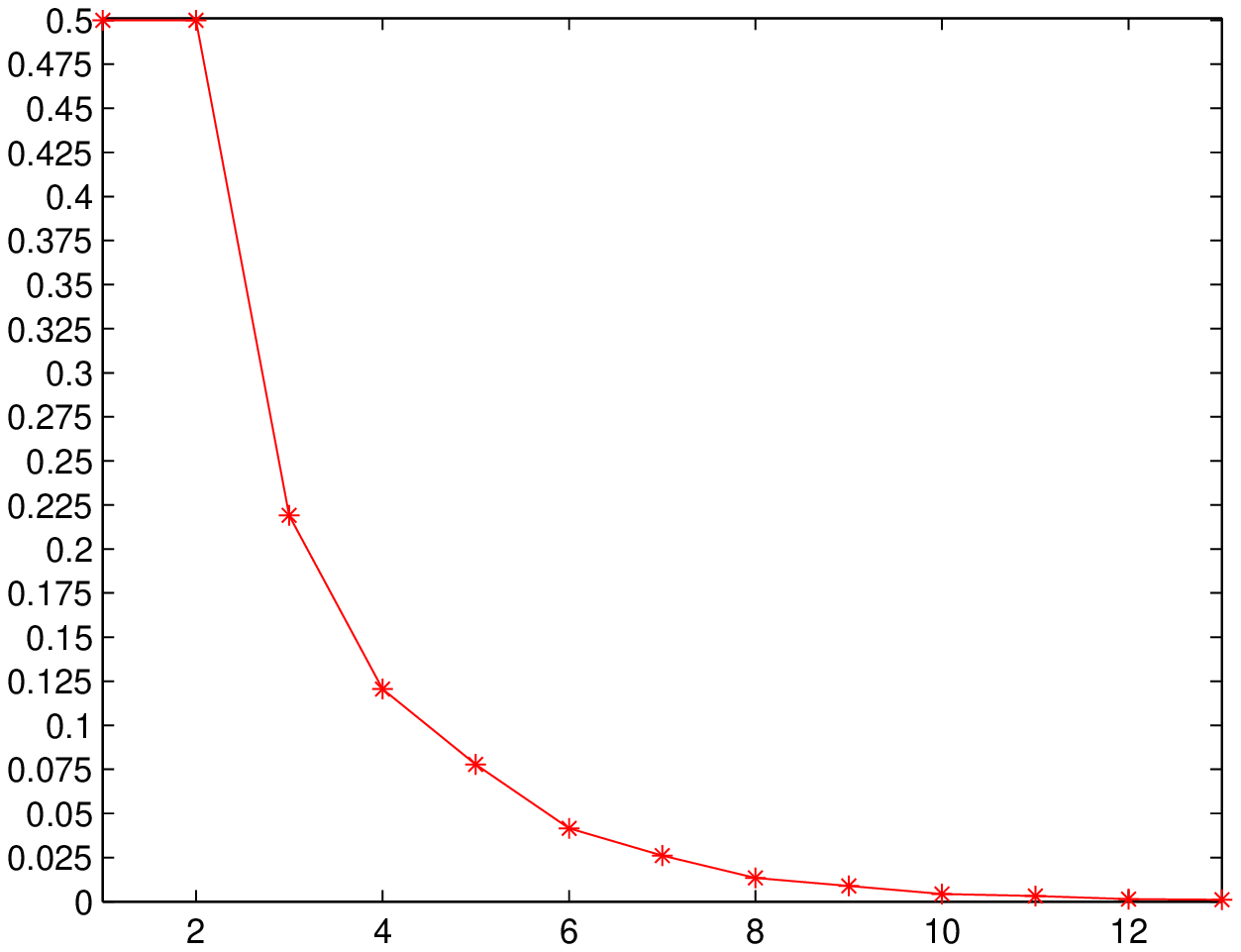}\hfill{}

        \hfill{}(c)\hfill{} \hfill{}(d)\hfill{}

        \hfill{}\includegraphics[clip,width=0.45\textwidth]{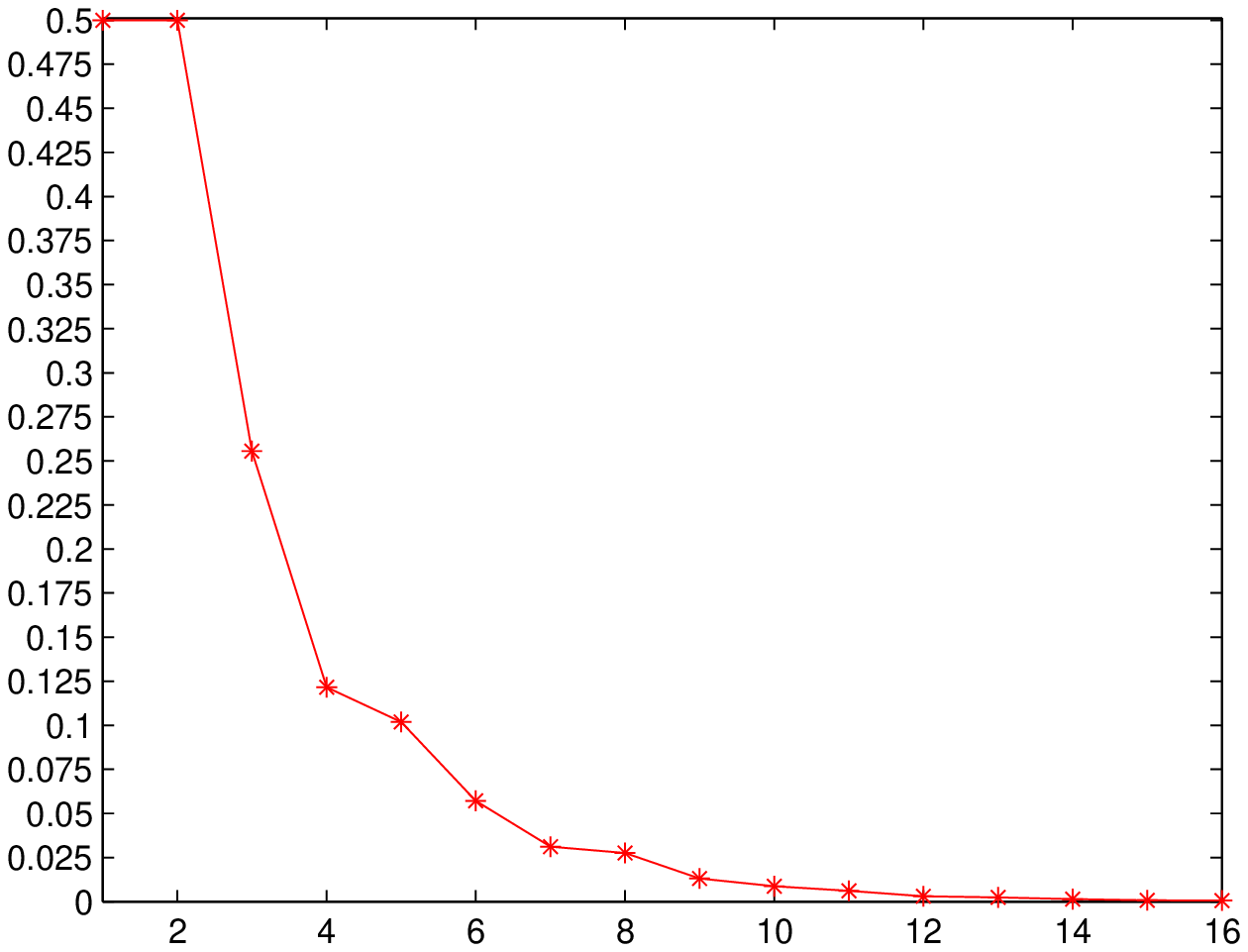}\hfill{}
        \hfill{}\includegraphics[clip,width=0.45\textwidth]{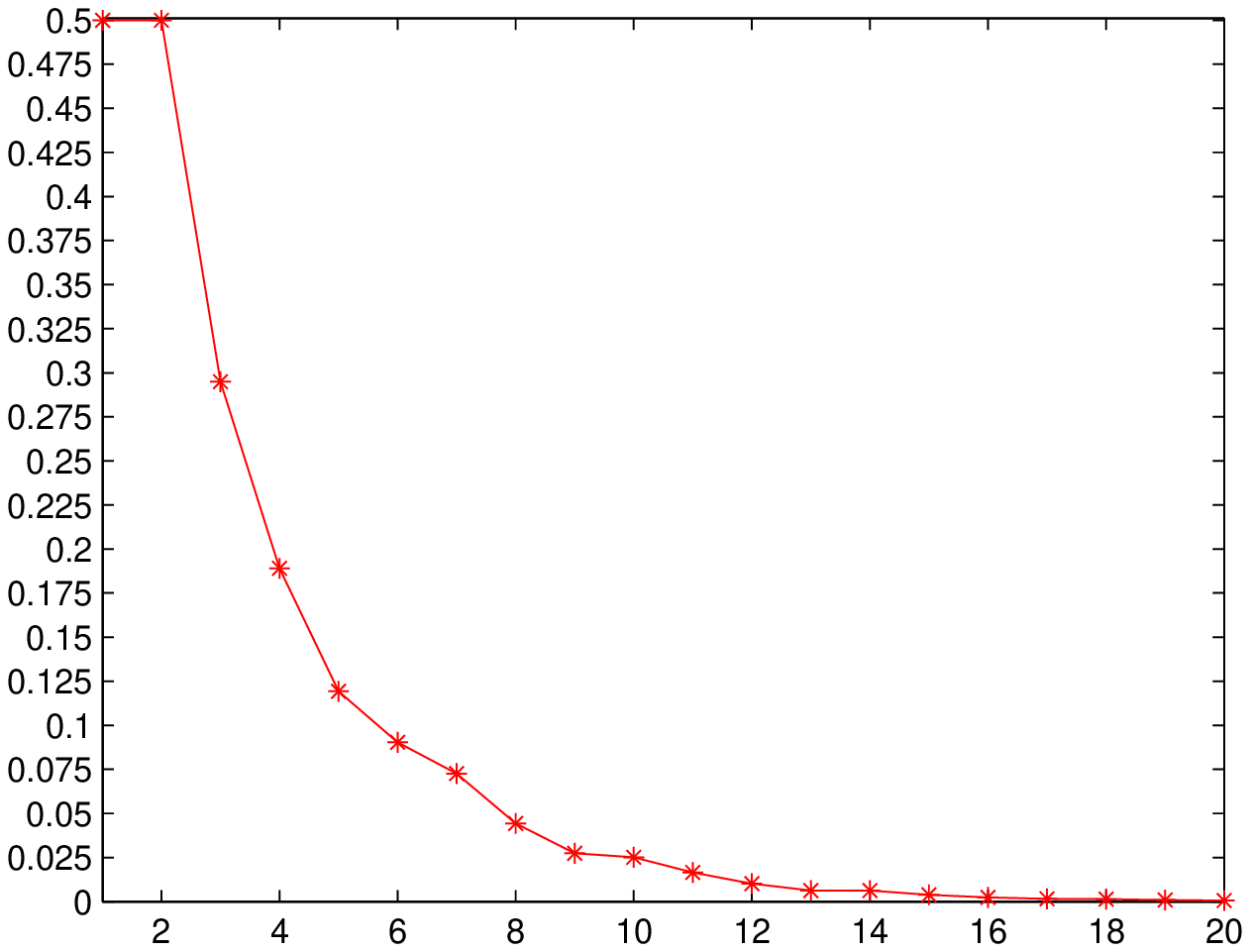}\hfill{}

        \hfill{}(e)\hfill{} \hfill{}(f)\hfill{}

        \hfill{}\includegraphics[clip,width=0.45\textwidth]{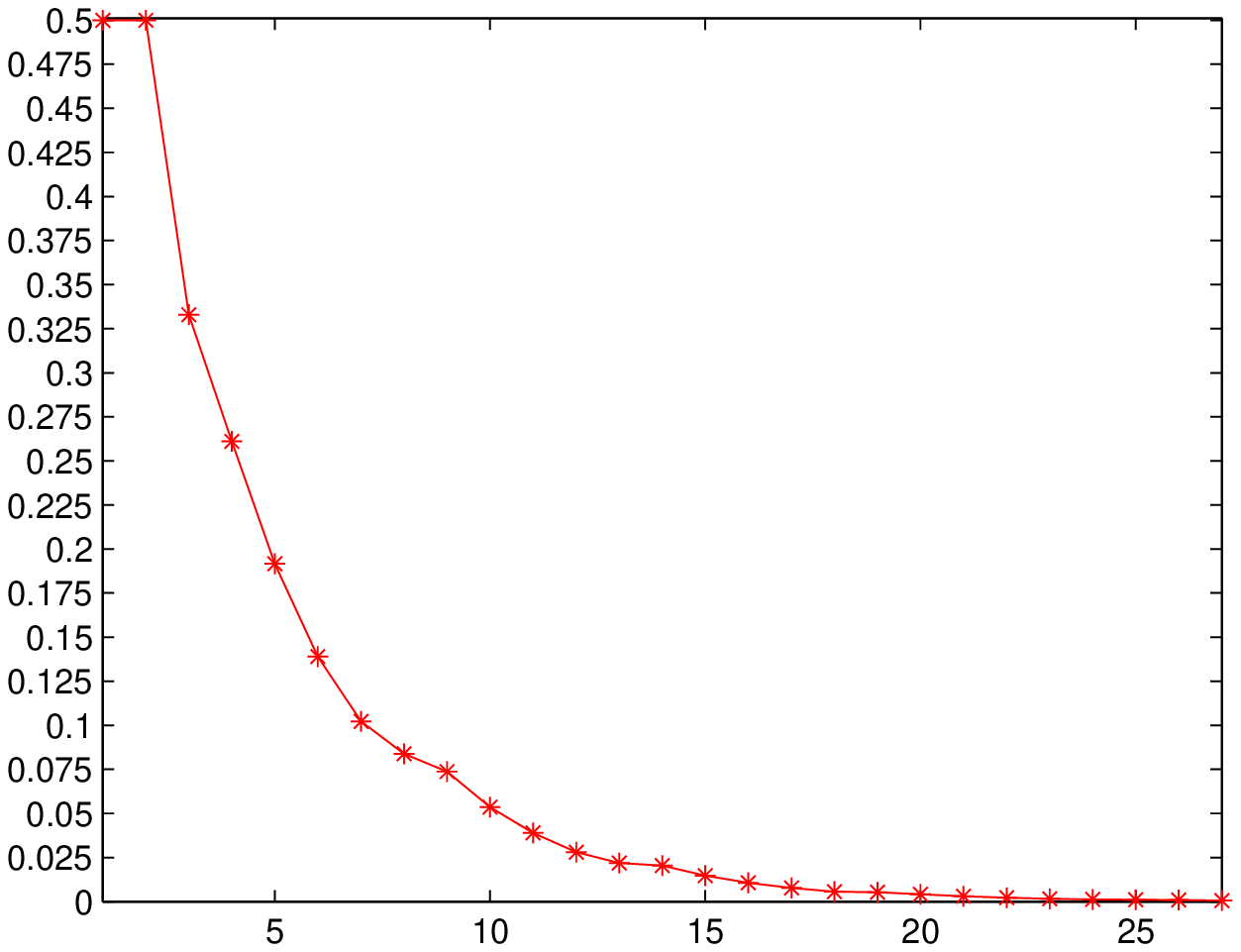}\hfill{}
        \hfill{}\includegraphics[clip,width=0.45\textwidth]{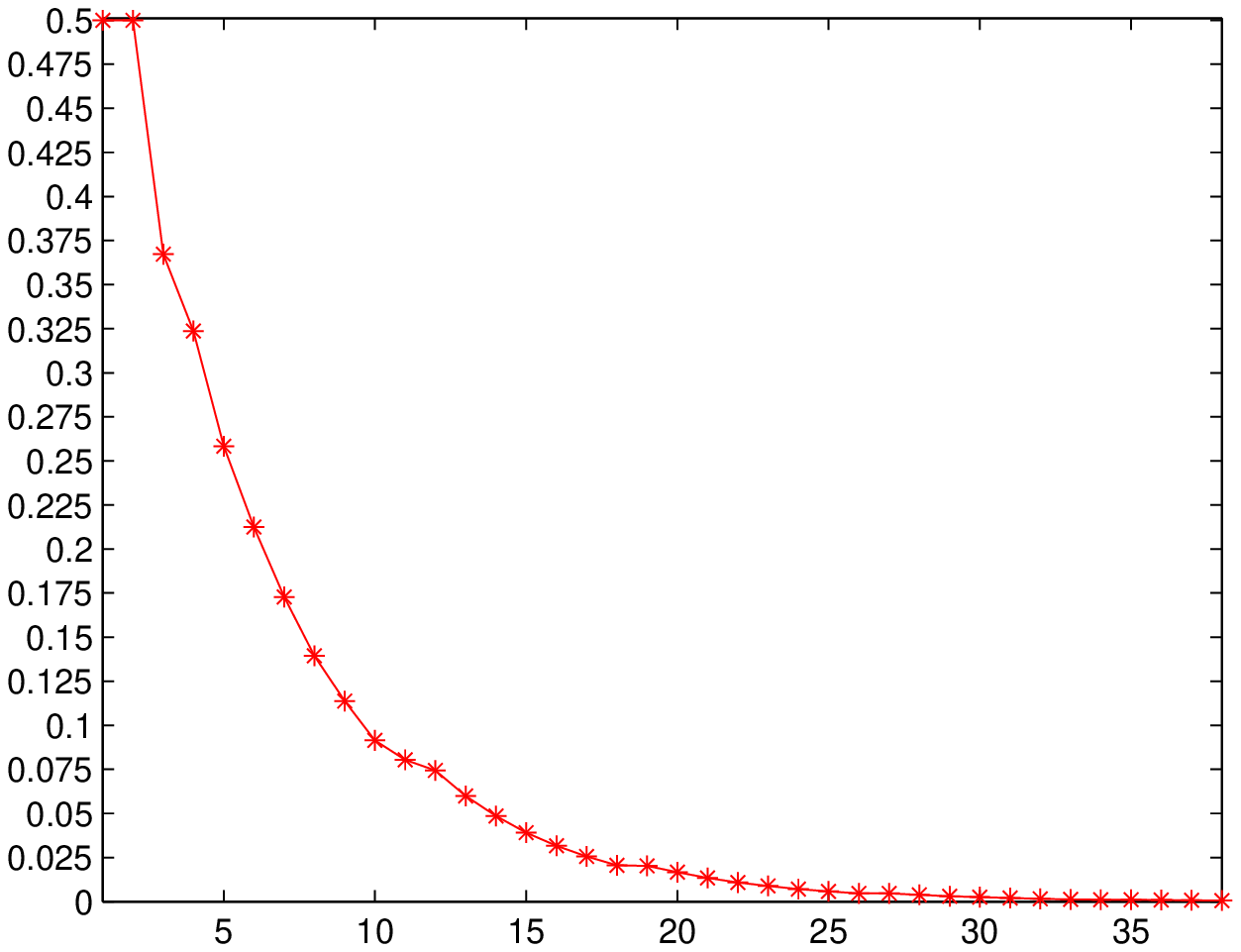}\hfill{}

        \hfill{}(g)\hfill{} \hfill{}(h)\hfill{}

        \hfill{}\includegraphics[clip,width=0.45\textwidth]{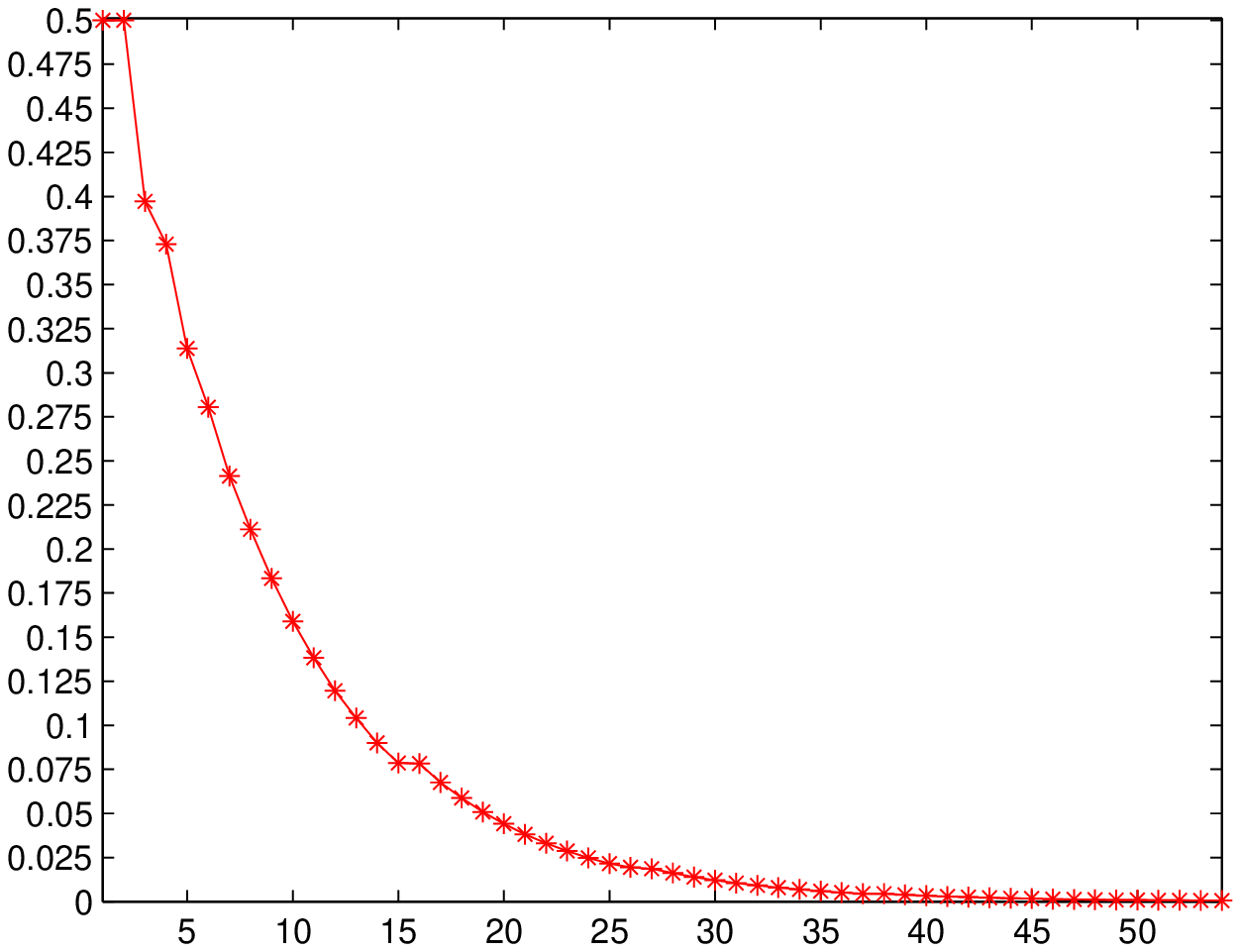}\hfill{}
        \hfill{}\includegraphics[clip,width=0.45\textwidth]{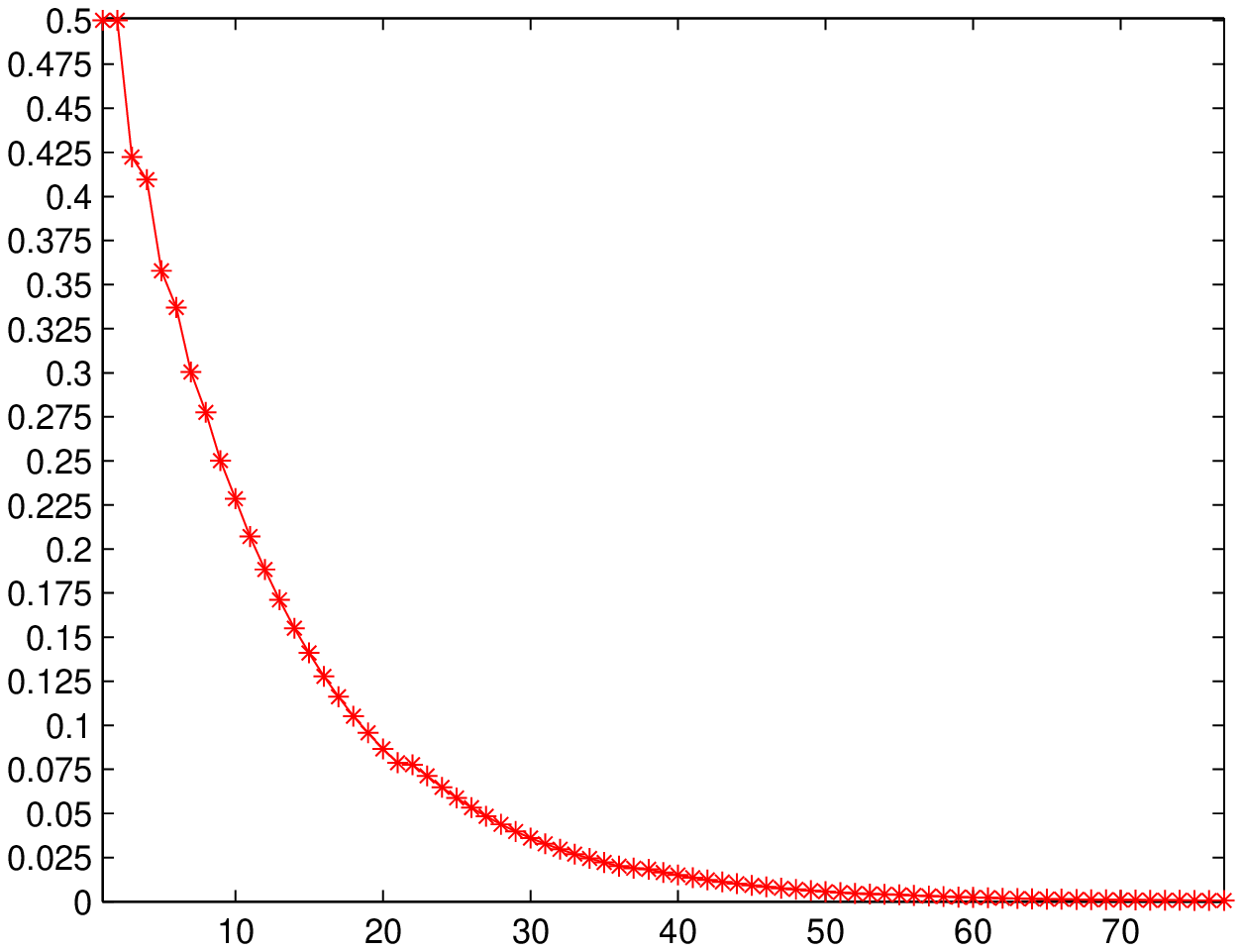}\hfill{}

        \hfill{}(i)\hfill{} \hfill{}(j)\hfill{}
        \vskip -0.2truecm
         \caption{\small Spectrum of $\mathbb{K}^*_{\partial D_v }$ in \eqref{matrixda} as $k=5-n$ with $n=1,2,\cdots,10$, starting from (a) with $n = 1$ to (j) with $n = 10$.}\label{multipletest}
\end{figurehere}

We note that the spectrum converges to a smoother curve where there are fewer ``steps".  Moreover, the multiplicity of the eigenvalue $1/2$ reflects the number of connected components of $D_v$; see \cite{ciraolo, triki2}.

\end{document}